\documentclass[11pt, reqno]{amsart}
\usepackage[utf8]{inputenc}
\usepackage{a4wide}
\usepackage{geometry}
\usepackage{amsmath,eucal, mathtools}
\usepackage{xcolor}
\usepackage{bbm,stix}
\usepackage{enumitem}
\usepackage{graphics}
\usepackage{epsfig}
\usepackage{graphicx}
\usepackage{wrapfig}
\usepackage{textcomp}
\usepackage{tikz}
\usepackage{pgfplots}
\usepackage{dsfont}
\pgfplotsset{compat=1.14}
\usepackage{comment}
\usepackage{multicol}
\usepackage{mathrsfs}
\usepackage{epstopdf}
\usepackage{cancel}
\usepackage[colorlinks=true,allcolors=black]{hyperref}
\usepackage{bm}
\usepackage[]{appendix}
\usepackage{booktabs} 
\usepackage{graphicx,subfig}
\usepackage{float}

\newtheorem{definition}{Definition}[section]
\newtheorem{theorem}[definition]{Theorem}
\newtheorem{prop}[definition]{Proposition}
\newtheorem{lem}[definition]{Lemma}
\newtheorem{cor}[definition]{Corollary}
\newtheorem{remark}[definition]{Remark}

\newcommand{\abs}[1]{\left|#1\right|}
\newcommand{\norm}[1]{\left\|#1\right\|}
\newcommand{\R}{\mathbb{R}}
\newcommand{\N}{\mathbb{N}}

\renewcommand{\d}{{\rm d}}
\newcommand{\de}{\partial}
\newcommand{\sgn}[1]{\mathrm{sign}\!\left(#1\right)}
\newcommand{\disint}[2]{\left\{ #1 , \ldots , #2 \right\}}
\newcommand{\vLip}{\left[v\right]_{\mathrm{Lip}}}
\newcommand{\TV}[1]{\mathrm{TV}\!\left[#1\right]}

\title[Fully discrete FtL approximation of scalar conservation laws with vacuum]{Fully discrete follow-the-leader approximation of one-dimensional scalar conservation laws with vacuum}

\author{M.\ Di Francesco, S.\ Fagioli, V.\ Iorio, M.D.\ Rosini}
\date{}

\begin{document}
\allowdisplaybreaks

\begin{abstract}
We present a fully discrete particle approximation for one-dimensional scalar conservation laws. Under suitable monotonicity assumptions on the macroscopic velocity, we construct a vacuum-compatible family of time-discrete particle equations and show that an appropriate piecewise-constant density reconstruction from the particle setting converges to the unique entropy weak solution of the macroscopic scalar conservation law.
\end{abstract}

\address{Marco Di Francesco - DISIM - Department of Information Engineering, Computer Science and Mathematics, University of L'Aquila, Via Vetoio 1 (Coppito)
67100 L'Aquila (AQ) - Italy}
\email{marco.difrancesco@univaq.it}

\address{Simone Fagioli - DISIM - Department of Information Engineering, Computer Science and Mathematics, University of L'Aquila, Via Vetoio 1 (Coppito)
67100 L'Aquila (AQ) - Italy}
\email{simone.fagioli@univaq.it}

\address{Valeria Iorio - DISIM - Department of Information Engineering, Computer Science and Mathematics, University of L'Aquila, Via Vetoio 1 (Coppito)
67100 L'Aquila (AQ) - Italy}
\email{valeria.iorio1@univaq.it}

\address{Massimiliano D. Rosini - Faculty of Mathematics and Computer Science, Maria Curie-Skłodowska University, Plac Marii Curie-Skłodowskiej 1, Lublin, 20031, Poland and
Department of Management and Business Administration, University \lq\lq G.~d'Annunzio\rq\rq\ of Chieti-Pescara, viale Pindaro, 42, Pescara, 65127, Italy
}
\email{massimiliano.rosini@unich.it}


\maketitle

\section{Introduction}

The approximation of entropy solutions to one-dimensional scalar conservation laws
by particle-based and Lagrangian methods has a long and rich history. Given an initial density $\overline{\rho} \colon \mathbb{R} \to [0,R]$, with $R>0$, we consider the Cauchy problem
\begin{equation}
\begin{dcases}\label{e:CP}
\partial_t \rho + \partial_x\!\left(\rho \, v(\rho)\right) = 0,\\
 \rho(x,0)=\overline{\rho}(x),
\end{dcases}
\end{equation}
where the velocity function $v \colon [0,R] \to\R$ is assumed to be strictly decreasing.
Equation~\eqref{e:CP} includes the classical Lighthill-Whitham-Richards model
as well as broader continuum descriptions of traffic and crowd dynamics, shallow water, sedimentation and particle settling and granular flow on an incline, see \cite{lighthill1955kinematic, richards1956shock, helbing2001traffic, whitham1974waves, dafermos2016hyperbolic, kynch1952sedimentation, savage1989motion, gray2001granular, Haberman-book, GaravelloPiccoli-book, Rosini-book} and references therein.

A natural Lagrangian discretization is given by the follow-the-leader (FTL) system
\begin{equation}
\begin{aligned}
&\dot x_N(t)= v(0),
\\
&\dot x_i(t) = v\!\left(\frac{\ell}{x_{i+1}(t)-x_i(t)}\right),&
&i\in \disint{0}{N-1},
\end{aligned}
\label{e:FTL}
\end{equation}
where $\ell = \norm{\overline{\rho}}_{L^1}/N$, with $N\in\N$. Note that fixing a separate dynamics for the rightmost particle is meant to capture the  \emph{upwind} nature of the passage from the macroscopic description to the particle-based one. In particular, the upwinding of the resulting scheme is \emph{not} determined by the sign of the velocity $v$ itself, that we can assume to be negative, but rather on its monotonicity. This becomes more clear by introducing the \emph{Lagrangian density}
\[
R(z,t) = \rho\bigl(X(z,t),t\bigr),
\]
where $X(\,\cdot\,,t)$ denotes the \emph{pseudo}-inverse of the cumulative distribution function
\[
F(x,t) = \int_{-\infty}^x \rho(y,t) \, \d{y}.
\]
Formally, rewriting the conservation law in these Lagrangian coordinates, one finds that $R$ satisfies the scalar conservation law
\[
\partial_t R + R^2 \, v'(R) \, \partial_z R = 0.
\]
When $v'(R) < 0$, the characteristic speeds of this equation are negative in the $z$-variable, implying that characteristics propagate towards decreasing values of $z$. In other words, the evolution of $R$ at a given Lagrangian position depends on values located \emph{ahead} in $z$. This shows that the effective upwinding in the Lagrangian formulation (and hence in the particle approximation) is dictated by the sign of $v'$, not by the sign of $v$ itself, see \figurename~\ref{fig:neg_vel} for a numerical example.

\begin{figure}
 \centering
 \includegraphics[width=7cm,height=6cm]{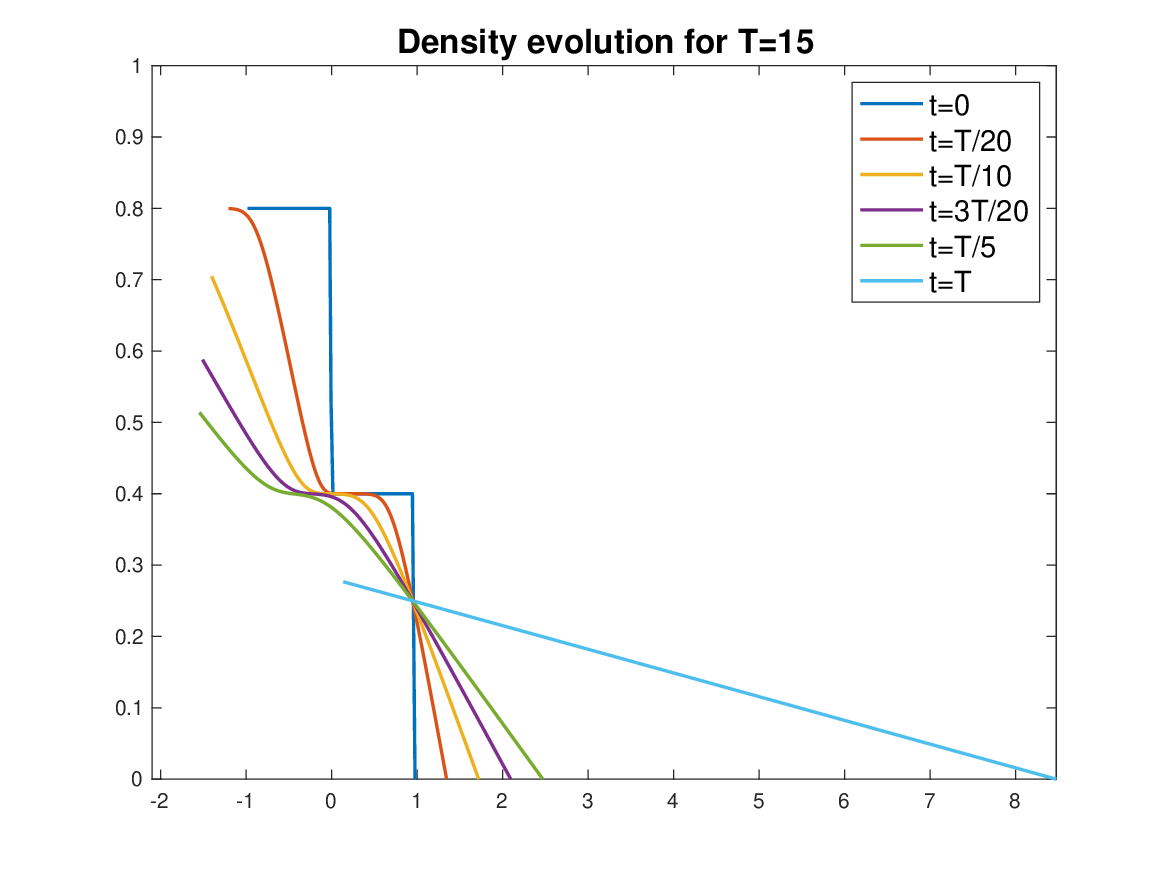}
 \includegraphics[width=7cm,height=6cm]{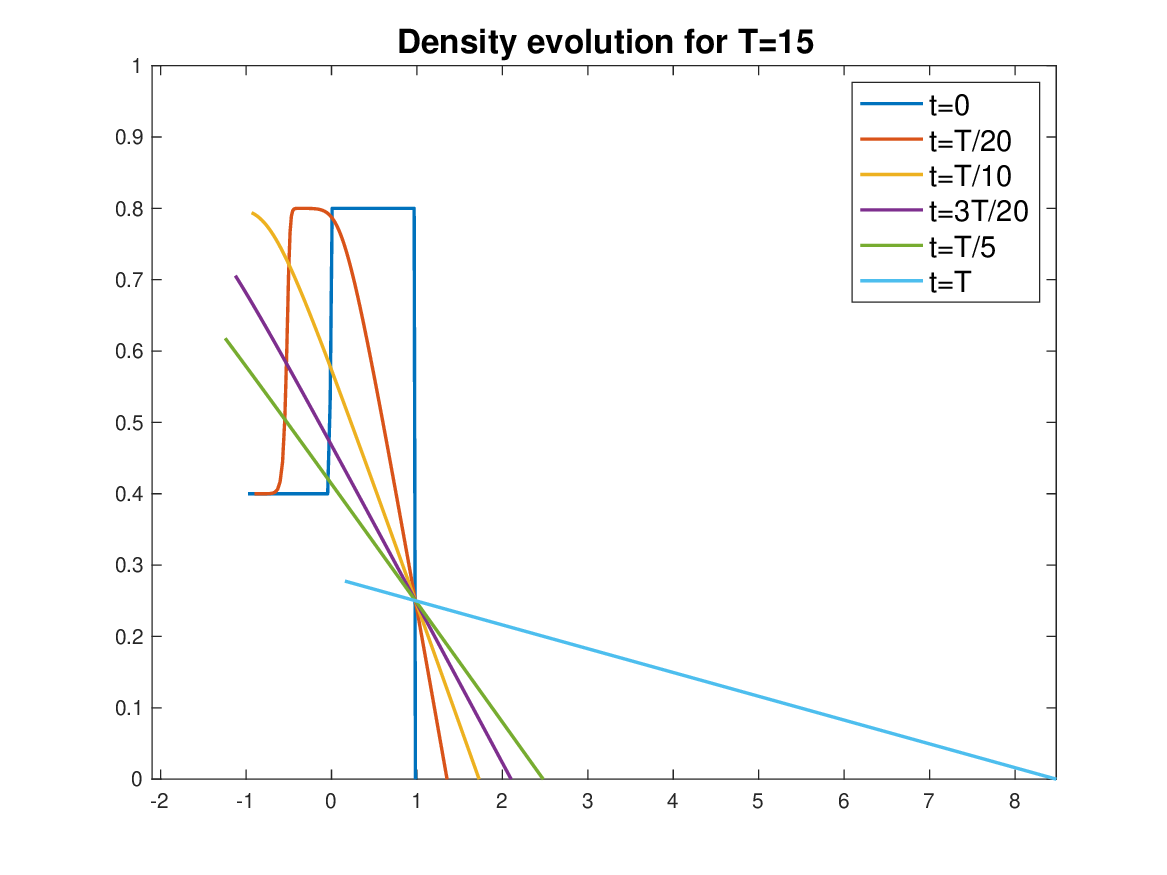}
\caption{Approximate solution $\rho_n$ plotted at different times, corresponding to
$v(\rho)=\frac{1}{2}-\rho$, with initial data
$\bar{\rho}(x)=0.8 \, \mathbf{1}_{[-1,0)}(x) + 0.4 \, \mathbf{1}_{[0,1)}(x)$ (left) and
$\bar{\rho}(x) = 0.4 \, \mathbf{1}_{[-1,0)}(x) + 0.8 \, \mathbf{1}_{[0,1)}(x)$ (right).
We use the implicit version ($\theta=0$) of the scheme described in \eqref{e:theta_intro} below, with $\ell=0.012$ and $\tau=0.01$.
The density $\rho_n$ attains values both greater than and less than $1/2$; hence the velocity $v(\rho_n)$ takes both negative and positive values.
Moreover, $\rho_n$ approximates the exact solution to \eqref{e:CP}, showing the correct upwinding in the scheme.
}
\label{fig:neg_vel}
\end{figure}

The rigorous connection between macroscopic and microscopic descriptions of scalar conservation laws has been extensively studied in the past decades. A probabilistic approach has been proposed in a vast literature,see e.g. \cite{ferrari, ferrari_TASEP, landim} and the references therein. Also a kinetic approximation of nonlinear conservation laws has been carried out in \cite{lions_perthame_tadmor}. Note that the wave-front tracking algorithm, originally introduced in \cite{dafermos1972polygonal, dafermos2016hyperbolic} (see also \cite{bressan1992wft, bressan2000hyperbolic, diperna1976global, holden2015front} and the references therein), is based on a micro-macro dichotomy as well, although of a different nature than the one considered in this work.

A different and highly influential perspective came with the many-particle limit theory, \cite{colombo_rossi, DiFrancescoRosini_ARMA}. In particular, in \cite{DiFrancescoRosini_ARMA}
authors established a rigorous derivation of nonlinear scalar conservation laws \eqref{e:CP} from the
microscopic FTL system \eqref{e:FTL}.
They proved that a proper discretized density associated with the
particle system converges to the unique Kruzhkov entropy solution of~\eqref{e:CP}.
Their analysis relies on discrete BV estimates and a discrete Oleĭnik-type inequality, already encoded at the discrete level. The aforementioned approach was also extended to other equations in different contexts, see \cite{AndreianovRosiniStivaletta, DiFrancescoFagioliRadici, DiFrancescoFagioliRosiniRusso1, DFIOSC2025, DiFrancescoStivaletta1, FagioliRadici, FagioliFavre, FaTse} and references therein.

These convergence results are based on a \emph{continuous-time} particle dynamics. In implementing the scheme numerically, a time discretization is unavoidable,
and explicit Euler approximations remain stable only under a strict positivity
assumption on the initial density, see \cite{holden2015front, holden2018ftl_lwr, holden2018continuum, holden2024continuum_nonlocal_ftl}.

The goal of this paper is to overcome these limitations by introducing a
\emph{fully discrete} many-particle approximation of~\eqref{e:CP}
via a general $\theta$-method:
\begin{equation}\label{e:theta_intro}
 x^{m+1}_i
= x^{m}_i
+ \tau \left(\theta \, v\!\left(\frac{\ell}{x^{m}_{i+1}-x^{m}_i}\right) + \left(1-\theta\right) v\!\left(\frac{\ell}{x^{m+1}_{i+1}-x^{m+1}_i}\right) \right),
\qquad i \in \disint{0}{N-1},
\end{equation}
together with a suitable condition for the rightmost particle.
Above, $\ell = \norm{\overline{\rho}}_{L^1}/N$, with $N \in \N$, $\tau$ denotes the time step, and $\theta \in [0,1]$ is the interpolation parameter between the implicit Euler ($\theta = 0$), the Crank--Nicolson ($\theta = 1/2$), and the explicit Euler ($\theta = 1$) schemes.
Such schemes are classical in the numerical treatment of ODEs and parabolic
equations \cite{QuarteroniValli2008, HairerWannerII, HundsdorferVerwer2003,
CrankNicolson1947}, but, up to the authors knowledge, their adaptation to particle-based approximations of
hyperbolic conservation laws was not investigated before,
especially in presence of vacuum.

The main contributions of this paper are the establishment of a vacuum-compatible
stability framework for the fully discrete particle system and the proof of
convergence of suitable piecewise constant approximations to entropy solutions
of~\eqref{e:CP}. In particular, we establish global well-posedness of the fully
discrete particle dynamics for all $\theta \in [0,1]$, under an appropriate
CFL-type condition, see \eqref{e:CFLtheta} below. This allows vacuum to be treated naturally, without
compromising the stability of the scheme. We then prove $L^\infty$ and BV estimates for a suitably reconstructed fully
discrete density, uniformly with respect to both the number of particles and the
number of time steps, thus extending the compactness theory
of~\cite{DiFrancescoRosini_ARMA} to the fully discrete setting. By combining
discrete stability, compactness arguments and Optimal Transport tools, we
show that a fully discrete density suitably constructed from the scheme \eqref{e:theta_intro} strongly converges in $L^1_{t,x}$ to the unique
Kruzhkov entropy solution of~\eqref{e:CP}, without requiring any strictly positive
lower bound on the initial density $\overline{\rho}$. We remark here that the convergence to weak solutions holds under the CFL-type condition, while a technical limitation requiring the time step $\tau$ to tend
to zero faster than the \emph{particle size} $\ell$ is imposed in order to obtain
convergence to the entropy solution. This requirement is more restrictive than the
standard CFL-type condition. While the natural hyperbolic scaling dictates the numerical relation $\tau =
\mathcal{O}(\ell)$, the identification of entropy solutions may require times going to zero faster than space, see for instance references relying on compactness arguments or
doubling-of-variables techniques \cite{kuznetsov1976accuracy,bressan2000hyperbolic}. The condition $\tau = o(\ell)$ in our context is a purely technical and analytical limitation, which is independent of the nature of the scheme. At first sight, the condition may seem to hide a parabolic scaling and thus suggest the presence of numerical viscosity. However, this is not possible because the scheme converges to the unique entropy solution, which rules out any artificial viscosity.

Overall, this work provides the first comprehensive fully discrete
Follow-the-Leader approximation of scalar conservation laws that is stable in
the presence of vacuum and compatible with a broad class of time
discretizations. In this way, it unifies and extends the approaches initiated
in~\cite{DiFrancescoRosini_ARMA,holden2018continuum}, offering a robust numerical
and analytical framework for equation considered.

The paper is organized as follows. In Section \ref{sec:main} we list all the required assumptions and we present the precise particle setting, together with its global well-posedness that is equivalent to the possibility of defining our discretized density, see Theorem \ref{thm:main_dis}. We then state our main result concerning the convergence to entropy weak solution in Theorem \ref{thm:main}. Sections \ref{s:proof_dis} and \ref{s:proof_main} are devoted to the proofs of the aforementioned results. In particular, in Sections \ref{s:proof_dis} we provide uniform $L^\infty$ and $BV$ bounds that allow to prove Theorem \ref{thm:main_dis}, whereas Section \ref{s:proof_main} contains all the ingredients for proving Theorem \ref{thm:main}. In Appendix \ref{sec:appendix} we provide a technical lemma concerning interpolation between $L^1$, $BV$ and $1$-Wasserstein spaces.

\section{The numerical particles scheme and main results}\label{sec:main}

In this section, we introduce our numerical scheme and prove that the associated discrete density is globally well defined, see Theorem~\ref{thm:main_dis}.
We also show that the scheme yields a sequence of discrete densities converging to the entropy solution of the Cauchy problem \eqref{e:CP}, see Theorem~\ref{thm:main}.

\subsection{The scheme}

Denote by $L^\infty_c(\R)$ the space of essentially bounded functions (w.r.t.~Lebesgue measure) with compact support.
Denote by $R>0$ the maximal density.
We shall work under the standing assumptions
\begin{gather}
\label{e:vel}
v\in C^1([0,R]),
\qquad
v'(\rho)<0 \quad \forall\rho\in [0,R],
\\
\label{e:inirho}
\overline{\rho} \in L^\infty_c(\R)\cap BV(\R),
\qquad
0 \leqslant \overline{\rho}(x) \leqslant R \hbox{ for a.e.~}x\in\R.
\end{gather}
Fix a time horizon $T>0$, two positive integers $M,N\in \N$, and a parameter $\theta \in [0,1]$.
We then set
\begin{equation}
\label{e:t_method_i}
x_{i}^{m+1} = x_{i}^{m} + \tau \left( \theta \, v\!\left(R_{i}^{m}\right) + \left(1-\theta\right) v\!\left(R_{i}^{m+1}\right) \right),
\quad i\in\disint{0}{N},\ m \in\disint{0}{M-1},
\end{equation}
where
\begin{equation}
\label{e:Rim}
R_{i}^{m} =
\begin{dcases}
\frac{\ell}{x_{i+1}^{m}-x_{i}^{m}} &\hbox{if } i \in \disint{0}{N-1},
\\
0&\hbox{if } i = N,
\end{dcases}
\qquad
\ m \in\disint{0}{M},
\end{equation}
and
\begin{align}
\label{e:tau-ell-L}
&\tau = T/M,&
&\ell = L / N,&
&L = \norm{\overline{\rho}}_{L^1} .
\end{align}
System \eqref{e:t_method_i} is coupled with the initial condition
\begin{equation}
\label{e:t_method_ini}
x_{i}^0 = \overline{x}_{i},\qquad i\in \disint{0}{N},
\end{equation}
where the initial positions are defined recursively as
\begin{equation}
\label{e:ini}
\left\{\begin{aligned}
\overline{x}_0 & =
\min\!\left(\mathrm{supp} (\overline{\rho})\right) ,
\\
\overline{x}_{i+1} & =
\min\left\{x\geqslant \overline{x}_{i}\,:\,\int_{\overline{x}_{i}}^x \overline{\rho}(y) \, \d{y} \geqslant \ell \right\} ,\qquad i\in\disint{0}{N-1} .
\end{aligned} \right.
\end{equation}
Observe that
\begin{equation}
\label{e:t_method_N}
x_{N}^{m+1} = x_N^{m}+\tau \, v(0) = \overline{x}_{N} + t^{m+1} \, v(0),\qquad m \in\disint{0}{M-1},
\end{equation}
where
\[t^{m} = m \, \tau.\]

\begin{remark}
In order to effectively solve the particle system \eqref{e:t_method_i}-\eqref{e:t_method_ini}, we observe that given a particle configuration $x_0^{m}<\ldots<x_N^{m}$ at time $t^{m} = m \, \tau$, the position of the particles at time $t^{m+1}$, $x_0^{m+1}<\ldots<x_N^{m+1}$, is derived from \eqref{e:t_method_i} and \eqref{e:t_method_ini}.
Indeed, the position $x_N^{m+1}$ is explicitly derived from the identity \eqref{e:t_method_N}. Then, the position $x_{i}^{m+1}$ is known once the positions $x_{i}^{m},x_{i+1}^{m}$, and $x_{i+1}^{m+1}$ are.
To see this, we write \eqref{e:t_method_i} as
\[x_{i}^{m+1}-\left(1-\theta\right) \tau \, v\!\left(\frac{\ell}{x_{i+1}^{m+1}-x_{i}^{m+1}}\right) = x_{i}^{m}+\theta \, \tau \, v\!\left(\frac{\ell}{x_{i+1}^{m}-x_{i}^{m}} \right) \]
and observe that the map
\begin{equation}
\label{e:map}
\left(-\infty,x^{m+1}_{i+1} - \frac{\ell}{R} \right]\ni x\mapsto x - \left(1-\theta\right) \tau \, v\!\left(\frac{\ell}{x_{i+1}^{m+1}-x} \right)
\end{equation}
is continuously differentiable with derivative
\[
1 - \left(1-\theta\right) \tau \, v'\!\left(\frac{\ell}{x_{i+1}^{m+1}-x} \right) \frac{\ell}{\left(x_{i+1}^{m+1}-x\right)^{2}} \geqslant 1
\]
in view of the assumption \eqref{e:vel}.
Hence, the map \eqref{e:map} is invertible and the position $x_{i}^{m+1}$ is well defined via the standard implicit function theorem.
\end{remark}

We denote by $\vLip$ the Lipschitz constant of $v$ and let
\begin{equation}
\label{e:V}
V = \norm{v}_{L^\infty} = \max\left\{\abs{v(0)},\abs{v(R)}\right\}.
\end{equation}
We emphasize that the velocity $v(\rho)$ can be negative, that is, $v(\rho)\in\R$. This assumption is not standard in the context of traffic flow modeling, but we prefer not to restrict ourselves to non-negative velocities in order to obtain more general results within a mathematical and modelling framework.

The discrete density associated with the scheme \eqref{e:t_method_i}-\eqref{e:t_method_ini} is defined on $\R\times[0,T)$ by
\begin{equation}
\label{e:disdens}
\rho^M_N(x,t) = \sum_{m=0}^{M-1} \sum_{i=0}^{N-1} R_{i}^{m} \, \mathbf{1}_{[x_{i}^{m},x_{i+1}^{m})}(x) \, \mathbf{1}_{[t^{m},t^{m+1})}(t)
\end{equation}
and is extended to $t=T$ by continuity. As shown in the next theorem, $\rho^M_N$ is globally well defined, see Section~\ref{s:proof_dis} for the proof.
\begin{theorem}
\label{thm:main_dis}
Consider a velocity function $v \colon [0,R] \to \R$ and an initial density $\overline{\rho} \colon \R \to [0,R]$ satisfying \eqref{e:vel} and \eqref{e:inirho}, respectively.
Fix a time horizon $T>0$, a parameter $\theta \in [0,1]$ and two positive integers $M,N\in\N$ satisfying the following CFL-type condition
\begin{equation}
\label{e:CFLtheta}
\theta \, N \, T \vLip R^2 < L \, M,
\end{equation}
where $L = \norm{\overline{\rho}}_{L^1(\R)}$.
In this case, the discrete density \eqref{e:disdens} is well defined on $\R\times[0,T]$ and satisfies the following estimates
\begin{align}
\label{e:MP}
(x,t) \in \R\times[0,T]\,
\Longrightarrow\,&
0 \leqslant \rho^M_N(x,t) \leqslant R,
\\
\label{e:TV}
m\in\disint{0}{M-1}\,
\Longrightarrow\,&
\TV{\rho^M_N\!\left(\,\cdot\,,t^{m+1}\right)} \leqslant \TV{\rho^M_N\!\left(\,\cdot\,,t^{m}\right)} \leqslant \TV{\overline{\rho}}.
\end{align}
Furthermore
\begin{equation}
\label{e:lipL1t}
0\leqslant t_1 \leqslant t_2 \leqslant T
\Longrightarrow
\norm{\rho^M_N(\,\cdot\,,t_2) - \rho^M_N(\,\cdot\,,t_1)}_{L^1} \leqslant C \, (t_2-t_1+\tau)^{1/2},
\end{equation}
where $\tau$ is given in \eqref{e:tau-ell-L}\textsubscript{1} and $C = 4 \left(2 \,\TV{\overline{\rho}} \, L \, V\right)^{1/2}$.
\end{theorem}

\subsection{Main result}

We are now in a position to state our main result, which shows that the discrete density \eqref{e:disdens} converges to the entropy solution of the Cauchy problem \eqref{e:CP}, see Section~\ref{s:proof_main}.
For completeness, we recall the definitions of weak and entropy solutions.

\begin{definition}
\label{d:weak}
A weak solution of the Cauchy problem \eqref{e:CP} is a continuous map $\rho \colon [0,T] \to L^\infty_c(\R)$ which satisfies
\begin{equation*}
\int_0^T \int_{\R} \left( \rho(x,t) \, \partial_t\varphi + f(\rho) \, \partial_x\varphi \right) \d{x} \, \d{t}
+ \int_{\R} \overline{\rho}(x) \, \varphi(x,0) \, \d{x}
= 0
\end{equation*}
for every test function $\varphi\in C_c^1(\R\times[0,T)])$, where $f(\rho) = \rho \, v(\rho)$.
\end{definition}

It is well known that weak solutions to scalar conservation laws are generally not unique. Uniqueness is recovered by imposing entropy
admissibility conditions, which select the physically relevant solution and
ensure stability with respect to perturbations of the initial data. In the
scalar case, entropy solutions are
unique in the sense of Kruzhkov; see, e.g.,
\cite{bressan2000hyperbolic, dafermos2016hyperbolic, kruvzkov}.

\begin{definition}
\label{d:entro}
An entropy solution of the Cauchy problem \eqref{e:CP} is a continuous map $\rho \colon [0,T] \to L^\infty_c(\R)$ which satisfies
\begin{equation}
\label{e:entro_ine}
\int_0^T \int_{\R} \left( \abs{\rho-k} \partial_t\varphi + \sgn{\rho-k} \left(f(\rho)-f(k)\right) \partial_x\varphi \right) \d{x} \, \d{t}
+ \int_{\R} \overline{\rho}(x) \, \varphi(x,0) \, \d{x}
\geqslant 0
\end{equation}
for every $k\in[0,R]$ and every non-negative test function $\varphi\in C_c^1(\R\times[0,T)])$, where $f(\rho) = \rho \, v(\rho)$.
\end{definition}

We have the following theorem, see Section~\ref{s:proof_main} for the proof.

\begin{theorem}
\label{thm:main}
Consider a velocity function $v \colon [0,R] \to \R$ and an initial density $\overline{\rho} \colon \R \to [0,R]$ satisfying \eqref{e:vel} and \eqref{e:inirho}, respectively.
Fix a time horizon $T>0$, a parameter $\theta \in [0,1]$ and two monotone sequences of positive integers $(M_n)_n,(N_n)_n\in\N$ satisfying the CFL-type condition 
\begin{equation}
\label{e:CFL}
\theta \, N_n \, T \vLip R^2 < L \, M_n,
\end{equation}
where $L = \norm{\overline{\rho}}_{L^1}$.
If
\begin{equation*}
\lim_{n\to+\infty}N_n = +\infty,
\end{equation*}
then the sequence of discrete densities $(\rho_n)_n$ with
\begin{equation}
\label{e:rn}
\rho_n = \rho^{M_n}_{N_n},
\end{equation}
see \eqref{e:disdens}, converges (up to a subsequence) to a weak solution of the Cauchy problem \eqref{e:CP}.
Moreover, if
\begin{equation}
\label{e:CFL-extra}
\lim_{n\to+\infty}\frac{N_n}{M_n} = 0,
\end{equation}
then the (whole) sequence of discrete densities $(\rho_n)_n$ converges to the unique entropy solution of the Cauchy problem \eqref{e:CP}.
\end{theorem}

\begin{remark}
In principle, it would suffice to show in Theorem~\ref{thm:main} that the limit is an entropy solution.
Indeed, by simply taking $k=0$ and $k=R$ in \eqref{e:entro_ine}, one can deduce that any entropy solution is also a weak solution.
However, the assumptions required to obtain a weak solution are less restrictive than those needed to recover the entropy solution, see \eqref{e:CFL-extra}.
For this reason, it is convenient to treat the two cases separately.
\end{remark}

\section{Well-posedness for the particle scheme and stability estimates}
\label{s:proof_dis}
This section is devoted to the proof of Theorem~\ref{thm:main_dis} and the estimates listed therein.

\subsection{Uniform bound for \texorpdfstring{$R_{i}^{m}$}{}}

In principle, the solution to \eqref{e:t_method_i}-\eqref{e:t_method_ini} may cease to exist after a finite number of iterations if any of the differences $x_{i+1}^{m} - x_{i}^{m}$ fall outside the interval $[\ell/R, +\infty)$. Indeed, in such cases, the quantity $R_{i}^{m}$ given in \eqref{e:Rim} would either be undefined or lie outside the domain of definition $[0,R]$ of the function $v$.
Hence, in order to prove a global existence for all discrete times $t^{m} = m \, \tau$, $m\in\disint{1}{M}$, we need to prove that the quantities $R_{i}^{m}$, $i\in \disint{0}{N-1}$, $m\in \disint{0}{M}$, belong to $[0,R]$, which is the goal of this section.

By assumption the CFL-type condition \eqref{e:CFLtheta} holds, namely
\begin{equation}
\label{e:CFLtheta-bis}
\theta \, N \, T \vLip R^2 < L \, M
\quad\Longleftrightarrow\quad
\left( \theta \vLip R^2 \right) \tau < \ell .
\end{equation}

For future use, we introduce
\begin{equation}
\label{e:dim}
d_{i}^{m} = x_{i+1}^{m}-x_{i}^{m} , \qquad i\in\disint{0}{N-1}.
\end{equation}
Observe that $R_{i}^{m} = \ell/d_{i}^{m}$ for all $i\in\disint{0}{N-1}$.

\begin{lem}
\label{l:Ri0esti}
$R_{i}^0 \in (0,R]$ for all $i \in\disint{0}{N-1}$.
\end{lem}
\begin{proof}
Fix $i\in\disint{0}{N-1}$.
By \eqref{e:Rim}, the definition of the initial positions \eqref{e:t_method_ini}, \eqref{e:ini}, and \eqref{e:inirho}\textsubscript{2}, we have
\[
R_{i}^0 =
\frac{\ell}{\overline{x}_{i+1}-\overline{x}_{i}} =
\fint_{\overline{x}_{i}}^{\overline{x}_{i+1}} \overline{\rho}(y) \, \d{y} \leqslant
\norm{\overline{\rho}}_{L^\infty} \leqslant R.
\]
At last, by \eqref{e:ini} we have $\overline{x}_{i+1} - \overline{x}_{i} >0$ because $\ell>0$ and $\overline{\rho} \in BV(\R)$ by \eqref{e:inirho}\textsubscript{1}; hence $R_{i}^0 >0$ by \eqref{e:Rim} and \eqref{e:t_method_ini}.
\end{proof}

\begin{lem} \label{l:estiNmu}
For all $m \in\disint{0}{M-1}$ we have
\begin{equation}
\label{e:estiNmu}
0 < R_{N-1}^{m+1} \leqslant R_{N-1}^{m}\leqslant R.
\end{equation}
\end{lem}
\begin{proof}
We extend the definition of $v$ to the entire interval $\R$ by setting $v(\rho) = v(0)$ for all $\rho < 0$ and $v(\rho) = v(R)$ for all $\rho > R$.
By Lemma~\ref{l:Ri0esti}, we know that $0 < R_{N-1}^0\leqslant R$.
Assume now
\[0 < R_{N-1}^{m} \leqslant R
\quad\Longleftrightarrow\quad
d_{N-1}^{m} \geqslant \ell/R
.\]
By \eqref{e:t_method_N} and \eqref{e:t_method_i}\textsubscript{$i = N-1$} we have
\begin{align*}
d_{N-1}^{m+1} & =
x_N^{m+1} - x_{N-1}^{m+1} =
\left( x_N^{m} + \tau \, v(0) \right) - \left( x_{N-1}^{m} + \theta \, \tau \, v\!\left(R_{N-1}^{m}\right) + \left(1-\theta\right) \tau \, v\!\left(R_{N-1}^{m+1}\right) \right) \\
& = d_{N-1}^{m} + \tau \left( v(0) - \theta \, v\!\left(R_{N-1}^{m}\right) - \left(1-\theta\right) v\!\left(R_{N-1}^{m+1}\right) \right) \\
& \geqslant d_{N-1}^{m},
\end{align*}
where the last estimate follows by observing that $v(\rho) \leqslant v(0)$ for any $\rho\in\R$ and recalling that $\theta \in [0,1]$.
As a consequence $d_{N-1}^{m+1} \geqslant d_{N-1}^{m} \geqslant \ell/R > 0$, and therefore $0 < R_{N-1}^{m+1} \leqslant R_{N-1}^{m}\leqslant R$.
\end{proof}

\begin{prop} \label{prop:max_princ}
Assume the CFL-type condition \eqref{e:CFLtheta-bis}.
Then we have
\begin{equation} \label{e:max_princ}
0 < R_{i}^{m}\leqslant R\qquad \hbox{for all $i\in\disint{0}{N-1}$ and $m\in\disint{0}{M}$.}
\end{equation}
\end{prop}
\begin{proof}
We extend the definition of $v$ to the entire interval $\R$ by setting $v(\rho) = v(0)$ for all $\rho < 0$ and $v(\rho) = v(R)$ for all $\rho > R$.
Observe that \eqref{e:max_princ} is equivalent to have $d_{i}^{m} \geqslant \ell/R$.
Recall that Lemma~\ref{l:Ri0esti} and Lemma~\ref{l:estiNmu} imply that $d_{i}^0 \geqslant \ell/R$ for all $i\in\disint{0}{N-1}$ and $d_{N-1}^{m} \geqslant \ell/R$ for all $m \in \disint{0}{M}$.
Suppose by contradiction that there exists $m+1\in \disint{1}{M}$ such that \eqref{e:max_princ} fails for some $i\in \disint{0}{N-2}$. Without restriction, we assume that $m+1$ is the first discrete time at which \eqref{e:max_princ} fails, and that $i$ is the largest index at which $d_{i}^{m+1} < \ell/R$.
From \eqref{e:t_method_i}, the monotonicity of $v$, and the hypothesis $d_{i}^{m+1} < \ell/R \leqslant d_{i+1}^{m+1}$, it follows that
\begin{equation}
\label{e:d_{i}_theta}
d_{i}^{m+1}  = d_{i}^{m} + \theta \, \tau \left( v\!\left(\ell/d_{i+1}^{m}\right) - v\!\left(\ell/d_{i}^{m}\right) \right) + \left(1-\theta\right) \tau \left( v(\ell/d_{i+1}^{m+1}) - v(\ell/d_{i}^{m+1}) \right)
\geqslant
d_{i}^{m} + \alpha \left(d_{i+1}^{m} - d_{i}^{m}\right),
\end{equation}
where
\[
\alpha = \theta \, \tau \, \frac{v\!\left(\ell/d_{i+1}^{m}\right) - v\!\left(\ell/d_{i}^{m}\right)}{d_{i+1}^{m} - d_{i}^{m}}.
\]
Since $v$ is decreasing, the map $d\mapsto v(\ell/d)$ is increasing and therefore
\[\alpha = \theta \, \tau \, \frac{v\!\left(\ell/d_{i+1}^{m} \right)-v\!\left(\ell/d_{i}^{m} \right)}{d_{i+1}^{m}-d_{i}^{m}} \geqslant 0.\]
We claim that $\alpha<1$.
Indeed, from $R_{i}^{m}, R_{i+1}^{m} \leqslant R$ and \eqref{e:CFLtheta-bis} we have
\begin{equation*}
\alpha =
\theta \, \tau \, \frac{v\!\left(R_{i+1}^{m}\right)-v\!\left(R_{i}^{m}\right)}{\frac{\ell}{R_{i+1}^{m}}-\frac{\ell}{R_{i}^{m}}}
\leqslant
\theta \, \tau \vLip \frac{\abs{R_{i+1}^{m}-R_{i}^{m}}}{\ell \abs{\frac{1}{R_{i+1}^{m}}-\frac{1}{R_{i}^{m}}}}
 =
\theta \, \frac{N \, T}{L \, M} \vLip R_{i}^{m} \, R_{i+1}^{m}
\leqslant
\theta \, \frac{N \, T}{L \, M} \vLip R^2
< 1.
\end{equation*}
Consequently, we have $\alpha \in [0,1)$.
From $\min\{d_{i}^{m}, d_{i+1}^{m} \} \geqslant \ell/R$ and \eqref{e:d_{i}_theta} we obtain
\begin{equation*}
d_{i}^{m+1} \geqslant d_{i}^{m} + \alpha \left(d_{i+1}^{m} - d_{i}^{m}\right)
= d_{i}^{m} \, (1-\alpha) + \alpha \, d_{i+1}^{m} \\
\geqslant \ell/R,
\end{equation*}
which contradicts our assumption $d_{i}^{m+1} < \ell/R$.
\end{proof}

As a consequence of Proposition~\ref{prop:max_princ}, the quantity $R_{i}^{m}$ is uniformly bounded by $R$ and strictly greater than zero.
Therefore, the particle system \eqref{e:t_method_i}-\eqref{e:t_method_ini} may be restarted at all discrete times $t^{m}$, $m\in\disint{0}{M-1}$, because the position of the particle $x_{i}^{m}$ is always far from the boundary of $(-\infty,x_{i+1}^{m})$ at least by the fixed quantity $\ell/R$. Consequently, the discrete density $\rho^M_N(x,t)$ defined in \eqref{e:disdens} is globally defined on $(x,t)\in \R\times [0,T]$ and satisfies the estimate \eqref{e:MP} by construction.

\subsection{\texorpdfstring{$BV$}{}-estimate}

The goal of this subsection is to provide a further uniform estimate on $\rho^M_N$ in some functional space that embeds compactly in $L^1(\R\times[0,T])$. The standard tool in numerical methods for conservation law is to estimate the total variation of $\rho^M_N$ in $x$. For future use, we set
\begin{equation} \label{e:rhoNm}
\rho_N^{m}(x) =
\rho^M_N\!\left(x,t^{m}\right) =
\sum_{i=0}^{N-1}R_{i}^{m} \, \mathbf{1}_{[x_{i}^{m},x_{i+1}^{m})}(x),
\qquad m\in \disint{0}{M}.
\end{equation}
For a fixed $m\in \disint{0}{M}$, the total variation (with respect to the space variable $x$) of $\rho_N^{m}$ is
\[
\TV{\rho_N^{m}} = R_0^{m}+\sum_{i=0}^{N-2} \abs{R_{i+1}^{m}-R_{i}^{m}} + R_{N-1}^{m}.
\]
Before stating our $BV$ estimate, for future use we recover below the discrete evolution law for $R_{i}^{m}$:
\begin{equation}
\label{e:Rdot1}
R_{i}^{m+1}-R_{i}^{m} =
-\tau \, \frac{R_{i}^{m} \, R_{i}^{m+1}}{\ell} \left[ \theta \left(v\!\left(R_{i+1}^{m}\right)-v\!\left(R_{i}^{m}\right)\right)+\left(1-\theta\right) \left(v\!\left(R_{i+1}^{m+1}\right)-v\!\left(R_{i}^{m+1}\right)\right) \right] ,
\end{equation}
for $i\in\disint{0}{N-1}$. We then provide a $BV$-bound for the reconstructed density in \eqref{e:rhoNm}.

\begin{prop} \label{prop:BV_estimate}
Assume the CFL-type condition \eqref{e:CFLtheta-bis} holds. Then,
\[\TV{\rho_N^{m+1}} \leqslant \TV{\rho_N^{m}} \leqslant \TV{\overline{\rho}}\]
for all $m\in \disint{0}{M-1}$.
\end{prop}

\begin{proof}
Since $\rho_{N}^{0} (\,\cdot\,,0)$ attains only some average values $R_{i}^0 = \fint_{x_{i}^0}^{x_{i+1}^0} \overline{\rho}(y) \, \d{y}$, $i \in \disint{0}{N-1}$, of $\overline{\rho}$, one can easily prove that $\TV{\rho_{N}^{0}} \leqslant \TV{\overline{\rho}}$.
We compute
\begin{align*}
\TV{\rho_N^{m+1}} - \TV{\rho_N^{m}} = &\left(R_0^{m+1} - R_0^{m}\right) + \left(R_{N-1}^{m+1} - R_{N-1}^{m}\right) + \sum_{i=0}^{N-2} \left( \abs{R_{i+1}^{m+1}-R_{i}^{m+1}} - \abs{R_{i+1}^{m}-R_{i}^{m}}\right)
\\ = {}
&\left(R_0^{m+1} - R_0^{m}\right) + \left(R_{N-1}^{m+1} - R_{N-1}^{m}\right)
+ \sum_{i=0}^{N-2} \sgn{R_{i+1}^{m+1}-R_{i}^{m+1}} (R_{i+1}^{m+1}-R_{i}^{m+1})
\\
&- \sum_{i=0}^{N-2} \sgn{R_{i+1}^{m}-R_{i}^{m}} \left(R_{i+1}^{m}-R_{i}^{m}\right)
\\ = {}
&\left(R_0^{m+1} - R_0^{m}\right) + \left(R_{N-1}^{m+1} - R_{N-1}^{m}\right)\\
&+ \sum_{i=0}^{N-2} \sgn{R_{i+1}^{m+1}-R_{i}^{m+1}} [(R_{i+1}^{m+1}-R_{i}^{m+1}) - \left(R_{i+1}^{m} - R_{i}^{m}\right)]
\\
&+\sum_{i=0}^{N-2} [\sgn{R_{i+1}^{m+1}-R_{i}^{m+1}} - \sgn{R_{i+1}^{m} - R_{i}^{m}}]\left(R_{i+1}^{m}-R_{i}^{m}\right)
\end{align*}
and thus
\begin{align*}
\TV{\rho_N^{m+1}} - \TV{\rho_N^{m}} = &\left(R_0^{m+1} - R_0^{m}\right) + \left(R_{N-1}^{m+1} - R_{N-1}^{m}\right)
+ \sum_{i=0}^{N-2} \sgn{R_{i+1}^{m+1}-R_{i}^{m+1}} \left(R_{i+1}^{m+1}-R_{i+1}^{m}\right)
\\
&- \sum_{i=0}^{N-2} \sgn{R_{i+1}^{m+1} - R_{i}^{m+1}} \left(R_{i}^{m+1}-R_{i}^{m}\right)\\
&+\sum_{i=0}^{N-2} \left[\sgn{R_{i+1}^{m+1}-R_{i}^{m+1}} - \sgn{R_{i+1}^{m} - R_{i}^{m}}\right]\left(R_{i+1}^{m}-R_{i}^{m}\right).
\end{align*}
Using summation by parts we get
\begin{align}
\TV{ \rho_N^{m+1}} - \TV{\rho_N^{m}} = {}&
\begin{aligned}[t]
&\left(R_0^{m+1} - R_0^{m}\right) + \left(R_{N-1}^{m+1} - R_{N-1}^{m}\right) \\
&+ \sgn{R_{N-1}^{m+1} - R_{N-2}^{m+1}} \left(R_{N-1}^{m+1} - R_{N-1}^{m}\right) -\sgn{R_1^{m+1} - R_0^{m+1}} \left(R_0^{m+1} - R_0^{m}\right) \\
&+ \left[ \sgn{R_1^{m+1} - R_0^{m+1}} - \sgn{R_1^{m} - R_0^{m}} \right] \left(R_1^{m} - R_0^{m}\right) \\
&+ \sum_{i=1}^{N-2} \left[\sgn{R_{i}^{m+1} - R_{i-1}^{m+1}} -\sgn{R_{i+1}^{m+1}-R_{i}^{m+1}} \right] (R_{i}^{m+1} -R_{i}^{m} ) \\
&+ \sum_{i=1}^{N-2} \left[\sgn{R_{i+1}^{m+1} - R_{i}^{m+1}} - \sgn{R_{i+1}^{m} - R_{i}^{m}} \right] \left(R_{i+1}^{m} - R_{i}^{m}\right) .
\end{aligned}
\label{e:comp_theta_2}
\end{align}
We set
\begin{align*}
A_{N-1} = {}& \left[ 1 + \sgn{R_{N-1}^{m+1} - R_{N-1}^{m+1}} \right] \left(R_{N-1}^{m+1} - R_{N-1}^{m}\right) ,
\\
A_0 = {}& \left[1-\sgn{R_1^{m+1} - R_0^{m+1}} \right] \left(R_0^{m+1} - R_0^{m}\right)+ \left[\sgn{R_1^{m+1} - R_0^{m+1}} - \sgn{R_1^{m} - R_0^{m}} \right] \left(R_1^{m} - R_0^{m}\right) .
\end{align*}
Observe that \eqref{e:estiNmu} implies $R_{N-1}^{m+1} - R_{N-1}^{m} \leqslant 0$ for all $m \in \disint{0}{M-1}$, which ensures that $A_{N-1} \leqslant 0$.
Introduce the notation
\begin{equation*}
\Delta v_{i}^{m} =
\begin{dcases}
\frac{v\!\left(R_{i+1}^{m}\right) - v\!\left(R_{i}^{m}\right)}{R_{i+1}^{m} - R_{i}^{m}}
& \hbox{if } R_{i+1}^{m} \neq R_{i}^{m} ,
\\
v'\!\left(R_{i}^{m}\right) & \hbox{otherwise} ,
\end{dcases}
\qquad i\in\disint{0}{N-2},\ m \in\disint{0}{M}.
\end{equation*}
Observe that $\Delta v_{i}^{m} \leqslant 0$, $i\in\disint{0}{N-2}$, $m \in\disint{0}{M}$, by the monotonicity of $v.$
Using \eqref{e:Rdot1} on the term $R_0^{m+1} - R_0^{m}$, $A_0$ can be rearranged as
\begin{align*}
A_0 = {}&
\left[1-\sgn{R_1^{m+1}-R_0^{m+1}} \right] \left[-\tau \, \frac{R_0^{m} \, R_0^{m+1}}{\ell} \right] \left[\theta \, \left(v\!\left(R_1^{m}\right) - v\!\left(R_0^{m}\right)\right) + \left(1-\theta\right) \left(v\!\left(R_1^{m+1}\right) - v\!\left(R_0^{m+1}\right)\right) \right]
\\
&+\left[\sgn{R_1^{m+1} - R_0^{m+1}} - \sgn{R_1^{m} - R_0^{m}} \right] \left(R_1^{m} - R_0^{m}\right)
\\
= {}&
\left[1-\sgn{R_1^{m+1}-R_0^{m+1}} \right] \left(\tau \, \frac{R_0^{m} \, R_0^{m+1}}{\ell} \right) \left(1-\theta\right) \left( - \Delta v_0^{m+1} \right) (R_1^{m+1} - R_0^{m+1})
\\
&+ \sgn{R_1^{m+1} - R_0^{m+1}} \left[1-\tau \, \frac{R_0^{m} \, R_0^{m+1}}{\ell} \, \theta \left( -\Delta v_0^{m}\right) \right] \left(R_1^{m} - R_0^{m}\right)
\\
&+ \left[ \tau \, \frac{R_0^{m} \, R_0^{m+1}}{\ell} \, \theta \left( -\Delta v_0^{m} \right) - \sgn{R_1^{m} - R_0^{m}} \right] \left(R_1^{m} - R_0^{m}\right).
\end{align*}
The first term in the last right-hand side above is non-positive because
\begin{align*}
&\left[1-\sgn{R_1^{m+1}-R_0^{m+1}} \right] \left(\tau \, \frac{R_0^{m} \, R_0^{m+1}}{\ell} \right) \left(1-\theta\right) \left( -\Delta v_0^{m+1} \right) (R_1^{m+1} - R_0^{m+1})
\\
= {}
&-2\left[R_0^{m+1}-R_1^{m+1} \right]_+ \left(\tau \, \frac{R_0^{m} \, R_0^{m+1}}{\ell} \right) \left(1-\theta\right) \left( - \Delta v_0^{m+1} \right)
\leqslant 0.
\end{align*}
Moreover, since $0\leqslant -\Delta v_0^{m+1} \leqslant \vLip$, by using the CFL condition \eqref{e:CFLtheta-bis} and the result in Proposition~\ref{prop:max_princ}, we deduce that
\[
1-\tau \, \frac{R_0^{m} \, R_0^{m+1}}{\ell} \, \theta \left( -\Delta v_0^{m}\right) \geqslant 1-\tau \, \frac{R^2}{\ell} \, \theta \vLip>0,
\]
and therefore the remaining terms are bounded as follows
\begin{align*}
&\sgn{R_1^{m+1} - R_0^{m+1}} \left[1-\tau \, \frac{R_0^{m} \, R_0^{m+1}}{\ell} \, \theta \left( -\Delta v_0^{m}\right) \right] \left(R_1^{m} - R_0^{m}\right)
\\
&+ \left[ \tau \, \frac{R_0^{m} \, R_0^{m+1}}{\ell} \, \theta \left( - \Delta v_0^{m} \right) - \sgn{R_1^{m} - R_0^{m}} \right] \left(R_1^{m} - R_0^{m}\right) \\
\leqslant{}& \left[1-\tau \, \frac{R_0^{m} \, R_0^{m+1}}{\ell} \, \theta \left( -\Delta v_0^{m}\right) \right] \abs{R_1^{m} - R_0^{m}}
+ \tau \, \frac{R_0^{m} \, R_0^{m+1}}{\ell} \, \theta \left( - \Delta v_0^{m} \right) \abs{R_1^{m} - R_0^{m}} - \abs{R_1^{m} - R_0^{m}}
= 0 .
\end{align*}
Hence, also $A_0\leqslant0$.

We now prove that the last two lines on the right-hand side of \eqref{e:comp_theta_2} are non-positive.
Due to \eqref{e:Rdot1}, we write
\begin{align*}
&\sum_{i=1}^{N-2} \left[\sgn{R_{i}^{m+1} - R_{i-1}^{m+1}}
- \sgn{R_{i+1}^{m+1}-R_{i}^{m+1}} \right](R_{i}^{m+1} -R_{i}^{m} ) \\
&+ \sum_{i=1}^{N-2} \left[\sgn{R_{i+1}^{m+1} - R_{i}^{m+1}} - \sgn{R_{i+1}^{m} - R_{i}^{m}} \right] \left(R_{i+1}^{m} - R_{i}^{m}\right) \\
= {}& \sum_{i=1}^{N-2} \sgn{R_{i+1}^{m+1} - R_{i}^{m+1}} \left( 1 - \tau \, \frac{R_{i}^{m} \, R_{i}^{m+1}}{\ell} \, \theta \left( - \Delta v_{i}^{m} \right) \right) \left(R_{i+1}^{m} - R_{i}^{m}\right) \\
&+ \sum_{i=1}^{N-2} \sgn{R_{i}^{m+1} - R_{i-1}^{m+1}} \, \tau \, \frac{R_{i}^{m} \, R_{i}^{m+1}}{\ell} \, \theta \left( - \Delta v_{i}^{m} \right) \left(R_{i+1}^{m} - R_{i}^{m}\right) \\
&- \sum_{i=1}^{N-2} \sgn{R_{i+1}^{m} - R_{i}^{m}} \left(R_{i+1}^{m} - R_{i}^{m}\right) \\
&+ \sum_{i=1}^{N-2} \left[ \left[ \sgn{R_{i}^{m+1} - R_{i-1}^{m+1}} - \sgn{R_{i+1}^{m+1} - R_{i}^{m+1}} \right] \tau \, \frac{R_{i}^{m} \, R_{i}^{m+1}}{\ell} \left(1-\theta\right) \left( - \Delta v_{i}^{m+1} \right) \right] \left(R_{i+1}^{m+1} - R_{i}^{m+1}\right).
\end{align*}
Now, since $0\leqslant -\Delta v_{i}^{m+1} \leqslant \vLip$, the CFL condition \eqref{e:CFLtheta-bis} and the result in Proposition~\ref{prop:max_princ} imply
\[
1 - \tau \, \frac{R_{i}^{m} \, R_{i}^{m+1}}{\ell} \, \theta \left( - \Delta v_{i}^{m} \right) \geqslant 1 - \tau \, \frac{R^2}{\ell} \, \theta \vLip > 0.
\]
Therefore, one easily gets
\begin{align*}
&\sum_{i=1}^{N-2} \sgn{R_{i+1}^{m+1} - R_{i}^{m+1}} \left( 1 - \tau \, \frac{R_{i}^{m} \, R_{i}^{m+1}}{\ell} \, \theta \left( - \Delta v_{i}^{m} \right) \right) \left(R_{i+1}^{m} - R_{i}^{m}\right) \\
&+ \sum_{i=1}^{N-2} \sgn{R_{i}^{m+1} - R_{i-1}^{m+1}} \, \tau \, \frac{R_{i}^{m} \, R_{i}^{m+1}}{\ell} \, \theta \left( - \Delta v_{i}^{m} \right) \left(R_{i+1}^{m} - R_{i}^{m}\right)
- \sum_{i=1}^{N-2} \sgn{R_{i+1}^{m} - R_{i}^{m}} \left(R_{i+1}^{m} - R_{i}^{m}\right) \\
\leqslant{}& \sum_{i=1}^{N-2} \left[
\left( 1 - \tau \, \frac{R_{i}^{m} \, R_{i}^{m+1}}{\ell} \, \theta \left( - \Delta v_{i}^{m} \right) \right) \abs{R_{i+1}^{m} - R_{i}^{m}}
+ \tau \, \frac{R_{i}^{m} \, R_{i}^{m+1}}{\ell} \, \theta \left( - \Delta v_{i}^{m} \right) \abs{R_{i+1}^{m} - R_{i}^{m}}
- \abs{R_{i+1}^{m} - R_{i}^{m}}
\right] = 0.
\end{align*}
Furthermore, we have
\[
\sum_{i=1}^{N-2} \left\{ \left[ \sgn{R_{i}^{m+1} - R_{i-1}^{m+1}} - \sgn{R_{i+1}^{m+1} - R_{i}^{m+1}} \right] \tau \, \frac{R_{i}^{m} \, R_{i}^{m+1}}{\ell} \left(1-\theta\right) \left( - \Delta v_{i}^{m+1} \right) \right\} \left(R_{i+1}^{m+1} - R_{i}^{m+1}\right)
\leqslant 0.
\]
This completes the proof.
\end{proof}

\begin{remark}
\emph{
In the implicit case $\theta = 0$, it holds that $\TV{\rho_N^{m+1}} \leqslant \TV{\rho_N^{m}}$ with no need of assuming any CFL-type condition. This is clear from the fact that \eqref{e:CFL} in that case is trivially satisfied.
}
\end{remark}

\subsection{Time oscillations and compactness}

We recall some tools from transport distances on probability measures. Let $\rho\in L^1(\R)\cap \mathcal{P}_c(\R)$, where $\mathcal{P}_c(\R)$ is the space of probability measures on $\R$ with compact support. We set $F_\rho \colon \R\rightarrow [0,L]$ as
\[F_\rho(x) = \int_{-\infty}^x \rho(y) \, \d{y}\]
and $X_\rho \colon (0,L)\rightarrow \R$ as the pseudo-inverse of $F_\rho$, given by
\[X_\rho(z) = \inf\{x\in \R\,:\, F_\rho(x)\geqslant z\} .\]
Given $\mu,\eta\in L^1(\R)\cap \mathcal{P}_c(\R)$, we recall (see \cite{villani_optimal} and \cite[(18)]{DiFrancescoRosini_ARMA}) that the scaled $1$-Wasserstein distance between $\mu$ and $\eta$ may be computed as
\[d_1(\rho,\eta) = \norm{X_\rho-X_\eta}_{L^1(0,L)} .\]

For notation simplicity, we define
\begin{align*}
F_N^{m} & = F_{\rho_N^{m}},&
X_N^{m} & = X_{\rho_N^{m}}.
\end{align*}

\begin{lem} \label{lem:representation}
For all $m\in \disint{0}{M}$, we have
\[X_N^{m}(z) = \sum_{i=0}^{N-1} \left(x_{i}^{m}+\frac{1}{R_{i}^{m}} \left(z-i \, \ell\right)\right) \mathbf{1}_{[i \, \ell,(i+1)\ell)} (z).\]
\end{lem}

\begin{proof}
From the definition of $\rho_N^{m}$ in \eqref{e:rhoNm}, we deduce that
\[F_N^{m}(x) = \sum_{i=0}^{N-1} \left(i \, \ell + R_{i}^{m}\left(x-x_{i}^{m}\right) \vphantom{\frac{1}{R_{i}^{m}}} \right) \mathbf{1}_{[x_{i}^{m},x_{i+1}^{m})}(x).\]
Since $F_N^{m}$ is piecewise linear, so is its pseudo-inverse $X_N^{m}$. Moreover, on each interval $(i \, \ell,(i+1) \ell)$ the slope of $X_N^{m}$ is the inverse of the slope of $F_N^{m}$ on the interval $(x_{i}^{m},x_{i+1}^{m})$. Imposing the continuity of $X_N^{m}$ at the points of discontinuity of its derivative implies the assertion.
\end{proof}

We start our compactness argument by estimating time-oscillations computing the $1$-Wasserstein distance between the linear-in-time interpolation at two separate times.

\begin{prop} \label{prop:time}
For all $0\leqslant t_1 \leqslant t_2 \leqslant T$ we have
\[ d_1\!\left(\rho_N^M(\,\cdot\,,t_2),\rho_N^M(\,\cdot\,,t_1)\right) \leqslant 4 L \, V \left(t_2-t_1+\tau\right).\]
\end{prop}

\begin{proof}
Fix $0\leqslant t_1 \leqslant t_2 \leqslant T$ and let $0\leqslant m_1 \leqslant m_2 \leqslant M$ be such that
\[t_1\in [t^{m_1},t^{m_1+1}),\qquad t_2\in [t^{m_2},t^{m_2+1}).\]
We use Lemma~\ref{lem:representation} to compute
\begin{align*}
& d_1\!\left(\rho_N^M(\,\cdot\,,t_2),\rho_N^M(\,\cdot\,,t_1)\right) =
d_1\!\left(\rho_N^{m_2},\rho_N^{m_1}\right) =
\norm{X_N^{m_2}-X_N^{m_1}}_{L^1(0,L)}
\\ = {}&
\sum_{i=0}^{N-1} \int_{i \, \ell}^{(i+1) \ell} \abs{x_{i}^{m_2}-x_{i}^{m_1}+\left(\frac{1}{R_{i}^{m_2}}-\frac{1}{R_{i}^{m_1}} \right)\left(z-i \, \ell\right)} \d{z}
\leqslant
\sum_{i=0}^{N-1} \left[\ell \abs{x_{i}^{m_2}-x_{i}^{m_1} } + \frac{\ell^2}{2} \abs{\frac{1}{R_{i}^{m_2}}-\frac{1}{R_{i}^{m_1}} }\right]\\
= {}& \sum_{i=0}^{N-1} \left[\ell \abs{x_{i}^{m_2}-x_{i}^{m_1} } + \frac{\ell}{2} \abs{(x^{m_2}_{i+1}-x^{m_2}_{i})-(x^{m_1}_{i+1}-x^{m_1}_{i})}\right]
\leqslant 2\ell \sum_{i=0}^{N} \abs{x_{i}^{m_2}-x_{i}^{m_1} }.
\end{align*}
Using the scheme \eqref{e:t_method_i}, we have
\begin{align*}
&x_{i}^{m_2}-x_{i}^{m_1} = \sum_{m = m_1}^{m_2-1} \left(x_{i}^{m+1}-x_{i}^{m}\right) =
\begin{dcases}
\tau\sum_{m = m_1}^{m_2-1} \left(\theta \, v\!\left(R_{i}^{m}\right)+\left(1-\theta\right) v\!\left(R_{i}^{m+1}\right)\right) & \hbox{if $i\in\disint{0}{N-1}$} ,\\
\tau (m_2-m_1)v(0) & \hbox{if $i = N$} .
\end{dcases}
\end{align*}
Hence, thanks to \eqref{e:MP} we have the estimate
\[\abs{ x_{i}^{m_2}-x_{i}^{m_1} } \leqslant V \, \tau \left(m_2-m_1\right)\]
with $V$ given in \eqref{e:V} and by \eqref{e:tau-ell-L} we get
\[
d_1\!\left(\rho_N^M(\,\cdot\,,t_2),\rho_N^M(\,\cdot\,,t_1)\right) \leqslant
2 \ell \, V \, \tau \left(m_2-m_1\right) (N+1) \leqslant
4 L \, V \left(t^{m_2}-t^{m_1}\right) \leqslant
4 L \, V \left(t_2-t_1+\tau\right).
\]
This concludes the proof.
\end{proof}

Next we improve the above estimate by computing the $L^1$-difference between $\rho^M_N(\,\cdot\,,t_2)$ and $\rho^M_N(\,\cdot\,,t_1)$ at two different times $0\leqslant t_1< t_2\leqslant T$.

\begin{cor}
For all $0\leqslant t_1 < t_2 \leqslant T$, we have
\[\norm{\rho^M_N(\,\cdot\,,t_2) - \rho^M_N(\,\cdot\,,t_1)}_{L^1(\R)} \leqslant 4 \left(2\TV{\overline{\rho}} L \, V\right)^{1/2} (t_2-t_1+\tau)^{1/2} .\]
\end{cor}

\begin{proof}
The proof follows easily by Proposition~\ref{prop:time}, Lemma~\ref{lem:appendix} and Proposition~\ref{prop:BV_estimate}.
\end{proof}

\section{Convergence to entropy solution}
\label{s:proof_main}
This subsection is devoted to the proof of Theorem~\ref{thm:main}.
\subsection{Compactness}

In the next Proposition we show the convergence (up to a subsequence) of the scheme.

\begin{prop}
\label{prop:convergence}
Consider a velocity function $v \colon [0,R] \to \R$ and an initial density $\overline{\rho} \colon \R \to [0,R]$ satisfying \eqref{e:vel} and \eqref{e:inirho}, respectively.
Fix a time horizon $T>0$, a parameter $\theta \in [0,1]$ and two monotone sequences of positive integers $(M_n)_n,(N_n)_n\in\N$ satisfying the CFL-type condition \eqref{e:CFL}
\[
\theta \, N_n \, T \vLip R^2 < L \, M_n.
\]
In this case, the sequence of discrete densities $(\rho_n)_n$ defined in \eqref{e:rn} converges up to a subsequence to some limit $\rho\in L^\infty_c(\R\times [0,T])$ strongly in $L^1(\R\times[0,T])$, and for all $t\in[0,T]$ it holds
\begin{align}
\label{e:SleepToken}
&\norm{\rho(\,\cdot\,,t)}_{L^\infty(\R)} \leqslant R,&
&\TV{\rho(\,\cdot\,,t)} \leqslant \TV{\overline{\rho}}.
\end{align}
\end{prop}

\begin{proof}
For any rational time $t\geqslant0$, by \eqref{e:MP} and \eqref{e:TV}, we can apply Helly's theorem, see \cite[Theorem~2.3]{bressan2000hyperbolic}, to construct a subsequence $(\rho_{n_k}(\,\cdot\,,t))_k$ that converges pointwise to some $\rho(\,\cdot\,,t)\in L^1_{loc}(\R)$ satisfying \eqref{e:SleepToken}.
By construction, if $[\overline{a},\overline{b}] = \mathrm{supp}(\overline{\rho})$ then for any $t\in[0,T]$ the support of $\rho_{n_k}(\,\cdot\,,t)$ is contained in $[\overline{a}- V T,\overline{b}+ V T]$, where $V$ is defined in \eqref{e:V}.
Therefore, $(\rho_{n_k}(\,\cdot\,,t))_k$ is dominated by $x \mapsto R \, \mathbf{1}_{[\overline{a}- V T,\overline{b}+ V T]}(x)$.
Hence, the aforementioned pointwise convergence implies compactness in $L^1(\R)$.
Now, for any $0\leqslant t_1 \leqslant t_2 \leqslant T$, from \eqref{e:lipL1t} we obtain
\[
\limsup_{k\rightarrow +\infty} \norm{ \rho_{n_k} (\,\cdot\,,t_2) - \rho_{n_k}(\,\cdot\,,t_1) }_{L^1(\R)} \leqslant C(t_2-t_1)^{1/2}
\]
with $C = 4 \left(2\TV{\overline{\rho}} L \, V\right)^{1/2}$ because $\tau_{n_k} = T/M_{n_k} \to 0$, see \eqref{e:tau-ell-L}\textsubscript{1}. Hence, a refined version of the Arzelà-Ascoli Theorem, see e.g.~\cite[Proposition~3.3.1]{AGS}, implies the assertion.
\end{proof}

\subsection{Convergence to a weak solution}

We are now in the position of proving that the limit $\rho$ obtained in Proposition~\ref{prop:convergence} is a weak solution to \eqref{e:CP} in the sense of Definition~\ref{d:weak}.

For all $i \in \disint{0}{N_n}$, we define the time-continuous interpolated curves $x_i \colon [0,T] \to \R$ as follows
\begin{equation}
\label{e:PinkFloyd}
x_{i}(t) =
\begin{dcases}
x_{i}^{m}+\frac{t-t^{m}}{\tau_n} \left(x_{i}^{m+1}-x_{i}^{m}\right) &\hbox{if } t \in (t^{m},t^{m+1}),\ m\in\disint{0}{M_n-1},
\\
x_i^{m}&\hbox{if } t=t^{m},\ m\in\disint{0}{M_n},
\end{dcases}
\end{equation}
where $\tau_n = T/M_n$.
We observe that by \eqref{e:t_method_i}, for all $m\in \disint{0}{M_n-1}$ and $i\in\disint{0}{N_n}$, the above curves satisfy
\begin{align}
\label{e:x_{i}_minus_x_m}
x_{i}(t)-x_{i}^{m} &= \frac{t-t^{m}}{\tau_n} \left(x_{i}^{m+1}-x_{i}^{m}\right)&
& \hbox{if } t\in [t^{m},t^{m+1}] ,
\\
\label{e:x_dot}
\dot{x}_{i}(t) &= \frac{x_i^{m+1}-x_i^{m}}{\tau_n} = \theta \, v\!\left(R_{i}^{m}\right)+\left(1-\theta\right) v\!\left(R_{i}^{m+1}\right)&
& \hbox{if }t\in (t^{m},t^{m+1}).
\end{align}
In particular, $\dot{x}_{N_n}(t) =v(0)$.
Moreover, we set
\begin{equation}
\label{e:Ri}
R_{i}(t) =
\begin{dcases}
\frac{\ell_n}{x_{i+1}(t)-x_{i}(t)} &\hbox{if } i \in \disint{0}{N_n-1},
\\
0&\hbox{if } i = N_n,
\end{dcases}
\end{equation}
where $\ell_n = L/N_n$.
Observe that $R_i(t^{m}) = R_i^{m}$, $i\in\disint{0}{N_n-1}$, for any $m \in \disint{0}{M_n}$.
Furthermore, by \eqref{e:x_{i}_minus_x_m}, \eqref{e:x_dot} and \eqref{e:Ri}, for all $m\in \disint{0}{M_n-1}$ and $i\in\disint{0}{N_n-1}$
\begin{align}
\label{e:R_i_minus_R_m}
R_{i}(t)-R_{i}^{m} &= -\frac{\tau_n}{\ell_n} \left(\frac{t-t^{m}}{\tau_n} \right) R_{i}(t) \, R_{i}^{m} \left(\dot{x}_{i+1}(t) - \dot{x}_{i}(t)\right)&
& \hbox{if } t\in (t^{m},t^{m+1}) ,
\\
\label{e:R_i_dot}
\dot{R}_{i}(t) &= -\frac{R_{i}(t)^2}{\ell_n} \left(\dot{x}_{i+1}(t) - \dot{x}_{i}(t)\right)&
& \hbox{if } t\in (t^{m},t^{m+1}) .
\end{align}

In the next lemma we list some properties that will be extensively used in the proofs of the next results.

\clearpage
\begin{lem}\,
\begin{enumerate}
\item
For any $m\in \disint{0}{M_n-1}$ and $i\in\disint{0}{N_n}$ we have
\begin{align}
\label{e:x_i_dot_est}
 t\in (t^{m}, t^{m+1}) \Longrightarrow& \abs{\dot{x}_{i}(t)} \leqslant V ,
\\
\label{e:x_i_minus_x_m_est}
t\in [t^{m}, t^{m+1}] \Longrightarrow& \abs{x_{i}(t)-x_{i}^{m}} \leqslant \abs{x_{i}^{m+1}-x_{i}^{m}} \leqslant \tau \, V ,
\end{align}
where $V$ is the constant defined in \eqref{e:V}.
\item
For any $i\in\disint{0}{N_n-1}$ there holds
\begin{equation}
\label{e:bound_R_i}
t\in[0,T] \Longrightarrow 0 < R_i(t) \leqslant R.
\end{equation}
\item
For any $m\in\disint{0}{M_n-1}$ and $i\in\disint{0}{N_n-1}$ there holds
\begin{equation}
\label{e:Ri-Rim_est}
t\in (t^{m}, t^{m+1}) \Longrightarrow \abs{R_{i}(t)-R_{i}^{m}} \leqslant \frac{\tau_n}{\ell_n} \vLip R_{i}(t) \, R_{i}^{m} \left[ \theta \abs{R_{i+1}^{m}-R_{i}^{m}} +\left(1-\theta\right) \abs{R_{i+1}^{m+1}-R_{i}^{m+1}} \right] .
\end{equation}
\item
For any $m\in\disint{0}{M_n-1}$ there holds
\begin{equation}
\label{e:dot_minus_dot}
t\in (t^{m}, t^{m+1}) \Longrightarrow \sum_{i=0}^{N_n-1}\abs{\dot{x}_{i+1}(t) - \dot{x}_i(t)} \leqslant \vLip \TV{\overline{\rho}}.
\end{equation}
\end{enumerate}
\end{lem}

\begin{proof}
Estimate \eqref{e:x_i_dot_est} is straightforward from \eqref{e:x_dot} and from Proposition~\ref{prop:max_princ}.
\eqref{e:x_i_minus_x_m_est} follows directly from \eqref{e:x_{i}_minus_x_m} and the definition of the scheme \eqref{e:t_method_i}, with the help of Proposition~\ref{prop:max_princ} too.

Fix $i\in\disint{0}{N_n-1}$. Take $d_i(t) = x_{i+1}(t)-x_i(t)$ and observe that
\[t \in [t^{m},t^{m+1}] \Longrightarrow 0< \ell_n/R \leqslant \min\{d_i^{m},d_i^{m+1}\} \leqslant d_i(t) \leqslant \max\{d_i^{m},d_i^{m+1}\}\]
by \eqref{e:PinkFloyd} and Proposition~\ref{prop:max_princ}, where $d_i^{m}$ is defined in \eqref{e:dim}.
Observe that $R_i(t) = \ell_n/d_i(t)$, hence
\[t \in [t^{m},t^{m+1}] \Longrightarrow 0< \min\{R_i^{m},R_i^{m+1}\} \leqslant R_i(t) \leqslant \max\{R_i^{m},R_i^{m+1}\} \leqslant R.\]
This estimate readily implies \eqref{e:bound_R_i}.
Then, \eqref{e:Ri-Rim_est} follows from \eqref{e:R_i_minus_R_m}, \eqref{e:x_dot} and \eqref{e:t_method_i}.
At last, \eqref{e:dot_minus_dot} can be obtained by \eqref{e:x_dot} and Proposition~\ref{prop:BV_estimate}.
\end{proof}

In the next proposition we show that the limit $\rho$ obtained in Proposition~\ref{prop:convergence} is indeed a weak solution of the Cauchy problem \eqref{e:CP} in the sense of Definition~\ref{d:weak}.

\begin{prop} \label{prop:weak}
Let $\rho_n = \rho_{N_n}^{M_n}$ with $(N_n)_n$, $(M_n)_n$ satisfying the CFL-type condition \eqref{e:CFLtheta} and such that $M_n,N_n \nearrow +\infty$.
Assume that $(\rho_n)_n$ converges to some limit $\rho\in L^\infty_c(\R\times [0,T])$ in $L^1(\R\times[0,T])$ and that it satisfies \eqref{e:SleepToken}.
Then, $\rho$ is a weak solution to the Cauchy problem \eqref{e:CP} on $\R\times [0,T]$.
\end{prop}

\begin{proof}
Let $\varphi\in C^\infty_0(\R\times [0,T))$. By construction, $\rho_n(\,\cdot\,,0)$ converges to $\overline{\rho}$ in the weak$^*$ measure sense, hence by \eqref{e:SleepToken} as $n\rightarrow +\infty$ we have
\begin{equation*}
\lim_{n\to+\infty} \mathcal{W}^n = \int_0^T \int_{\R} \rho(x,t) \left(\partial_t\varphi(x,t) + v\!\left(\rho(x,t)\right) \partial_x\varphi(x,t)\right)\d{x} \, \d{t}+\int_\R \overline{\rho}(x) \, \varphi(x,0) \, \d{x} ,
\end{equation*}
where
\[\mathcal{W}^n = \int_0^T \int_{\R} \rho_n(x,t)\left(\partial_t\varphi(x,t) + v\!\left(\rho_n(x,t)\right) \partial_x\varphi(x,t)\right)\d{x} \, \d{t}+\int_\R \rho_n(x,0) \, \varphi(x,0) \, \d{x} .\]
In order to prove that the limit $\rho$ is a weak solution to the Cauchy problem \eqref{e:CP}, it suffices to prove that $\mathcal{W}^n$ converges to zero as $n\rightarrow+\infty$. We use \eqref{e:x_dot} and get
\begin{align*}
\mathcal{W}^n = {}&
\sum_{i=0}^{N_n-1} \left[ \sum_{m=0}^{M_n-1} \int_{t^{m}}^{t^{m+1}}R_{i}^{m}\left(\int_{x_{i}^{m}}^{x_{i+1}^{m}} \partial_t\varphi(x,t) \, \d{x} +v\!\left(R_{i}^{m}\right) \left(\varphi\!\left(x_{i+1}^{m},t\right) - \varphi\!\left(x_{i}^{m},t\right)\right) \right) \d{t}\right.\\
& \qquad \left.
+ R_{i}^0\int_{x_{i}^0}^{x_{i+1}^0} \varphi(x,0) \, \d{x} \right]
\\ = {}&
\mathfrak{M}^n+\mathfrak{E}^n_1+\mathfrak{E}^n_2+\mathfrak{E}^n_3+\mathfrak{E}^n_4,
\end{align*}
with the main term
\begin{align*}
\mathfrak{M}^n = {}&\sum_{m=0}^{M_n-1} \sum_{i=0}^{N_n-1} \int_{t^{m}}^{t^{m+1}}R_{i}(t) \left(\int_{x_{i}(t)}^{x_{i+1}(t)} \partial_t\varphi(x,t) \, \d{x} +\dot{x}_{i}(t) \left( \varphi\!\left(x_{i+1}(t),t\right) - \varphi\!\left(x_{i}(t),t\right) \right) \right) \d{t} \\
&+\sum_{i=0}^{N_n-1}R_{i}^0\int_{x_{i}^0}^{x_{i+1}^0} \varphi(x,0) \, \d{x} ,
\end{align*}
and the error terms
\begin{align*}
\mathfrak{E}^n_1 & = \sum_{m=0}^{M_n-1} \sum_{i=0}^{N_n-1} \int_{t^{m}}^{t^{m+1}} (R_{i}^{m}-R_{i}(t)) \left(\int_{x_{i}^{m}}^{x_{i+1}^{m}} \partial_t\varphi(x,t) \, \d{x} + \dot{x}_{i}(t) \left( \varphi\!\left(x_{i+1}^{m},t\right) - \varphi\!\left(x_{i}^{m},t\right)\right) \right) \d{t} ,
\\
\mathfrak{E}^n_2 & = \sum_{m=0}^{M_n-1} \sum_{i=0}^{N_n-1} \int_{t^{m}}^{t^{m+1}}R_{i}(t) \left(\int_{x_{i}^{m}}^{x_{i+1}^{m}} \partial_t\varphi(x,t) \, \d{x}-\int_{x_{i}(t)}^{x_{i+1}(t)} \partial_t\varphi(x,t) \, \d{x}\right) \d{t} ,
\\
\mathfrak{E}^n_3 & = \sum_{m=0}^{M_n-1} \sum_{i=0}^{N_n-1} \int_{t^{m}}^{t^{m+1}}R_{i}(t) \, \dot{x}_{i}(t) \left(\varphi\!\left(x_{i+1}^{m},t\right) - \varphi\!\left(x_{i}^{m},t\right) - \varphi\!\left(x_{i+1}(t),t\right)+\varphi\!\left(x_{i}(t),t\right)\right) \d{t} ,
\\
\mathfrak{E}^n_4 & = \sum_{m=0}^{M_n-1} \sum_{i=0}^{N_n-1} \int_{t^{m}}^{t^{m+1}} R_{i}^{m} \left( v\!\left(R_{i}^{m}\right)-\dot{x}_{i}(t) \right) (\varphi\!\left(x_{i+1}^{m},t\right) - \varphi\!\left(x_{i}^{m},t\right)) \, \d{t} .
\end{align*}
We compute the main term $\mathfrak{M}^n$ using \eqref{e:PinkFloyd}, \eqref{e:Ri}, integration by parts in time, \eqref{e:R_i_dot} and \eqref{e:x_dot} to obtain
\begin{align*}
\mathfrak{M}^n {} = &
\int_0^T \sum_{i=0}^{N_n-1} R_{i}(t) \left( \int_{x_{i}(t)}^{x_{i+1}(t)} \partial_t\varphi(x,t) \, \d{x} +\dot{x}_{i}(t)(\varphi\!\left(x_{i+1}(t),t\right) - \varphi\!\left(x_{i}(t),t\right)) \right) \d{t}
+\sum_{i=0}^{N_n-1}R_{i}(0)\int_{x_{i}(0)}^{x_{i+1}(0)} \varphi(x,0) \, \d{x}
\\{} = &
\int_0^T \sum_{i=0}^{N_n-1}R_{i}(t)
\begin{aligned}[t]
\left[ \vphantom{\int_{x_{i}(t)}^{x_{i+1}(t)}} \right.
&
\frac{\d}{\d{t}} \left(\int_{x_{i}(t)}^{x_{i+1}(t)} \varphi(x,t) \, \d{x}\right) - \left(\varphi\!\left(x_{i+1}(t),t\right)\dot{x}_{i+1}(t) - \varphi\!\left(x_{i}(t),t\right) \dot{x}_{i}(t)\right)
\\
&+\dot{x}_{i}(t)(\varphi\!\left(x_{i+1}(t),t\right) - \varphi\!\left(x_{i}(t),t\right))\Biggr] \, \d{t} +\sum_{i=0}^{N_n-1}R_{i}(0)\int_{x_{i} (0)}^{x_{i+1} (0)} \varphi(x,0) \, \d{x}
\end{aligned}
\\{} = &
\int_0^T \sum_{i=0}^{N_n-1} \left[-\dot{R}_{i}(t) \int_{x_{i}(t)}^{x_{i+1}(t)} \varphi(x,t) \, \d{x} +R_{i}(t) \, \varphi\!\left(x_{i+1}(t),t\right)(\dot{x}_{i}(t) - \dot{x}_{i+1}(t))\right] \d{t}
\\{} = &
\int_0^T \sum_{i=0}^{N_n-1} \left[\frac{R_{i}(t)^2}{\ell_n} (\dot{x}_{i+1}(t)-\dot{x}_{i}(t))\int_{x_{i}(t)}^{x_{i+1}(t)} \left(\varphi(x,t) - \varphi\!\left(x_{i+1}(t),t\right)\right) \d{x}\right] \d{t}
\\{} = &
\sum_{m=0}^{M_n-1} \int_{t^{m}}^{t^{m+1}} \sum_{i=0}^{N_n-1}
\begin{aligned}[t]
\left[ \vphantom{\int_{x_{i}(t)}^{x_{i+1}(t)}} \right.
&
\frac{R_{i}(t)^2}{\ell_n} \left( \theta \left( v\!\left(R_{i+1}^{m}\right) - v\!\left(R_{i}^{m}\right) \right) + \left(1-\theta\right) \left( v\!\left(R_{i+1}^{m+1}\right) - v\!\left(R_{i}^{m+1}\right) \right) \right)
\\&\left.
\times \int_{x_{i}(t)}^{x_{i+1}(t)} \left(\varphi(x,t) - \varphi\!\left(x_{i+1}(t),t\right)\right) \d{x} \right] \d{t}.
\end{aligned}
\end{align*}
Observe that for all $i\in\disint{0}{N_n-1}$
\begin{align*}
&\int_{t^{m}}^{t^{m+1}} \frac{R_{i}(t)^2}{\ell_n} \abs{ \int_{x_{i}(t)}^{x_{i+1}(t)} \left(\varphi(x,t) - \varphi\!\left(x_{i+1}(t),t\right)\right) \d{x} } \d{t}
\\\leqslant{}&
\int_{t^{m}}^{t^{m+1}} \frac{R_{i}(t)^2}{\ell_n} \left( \frac{1}{2} \norm{\partial_x\varphi}_{L^\infty} \left(x_{i+1}(t)-x_{i}(t)\right)^2 \right) \d{t}
=
\int_{t^{m}}^{t^{m+1}} \frac{1}{2} \norm{\partial_x\varphi}_{L^\infty} \ell_n \, \d{t}
=
\frac{1}{2} \norm{\partial_x\varphi}_{L^\infty} \ell_n \, \tau_n .
\end{align*}
Hence, using Proposition~\ref{prop:BV_estimate} we get
\begin{align*}
\abs{\mathfrak{M}^n} &\leqslant
\frac{1}{2} \norm{\partial_x\varphi}_{L^\infty} \ell_n \tau_n \sum_{m=0}^{M_n-1} \sum_{i=0}^{N_n-1} \left( \theta \abs{v\!\left(R_{i+1}^{m}\right) - v\!\left(R_{i}^{m}\right) } + \left(1-\theta\right) \abs{ v\!\left(R_{i+1}^{m+1}\right) - v\!\left(R_{i}^{m+1}\right) } \right)
\\
& \leqslant
\left( \frac{1}{2} \norm{\partial_x\varphi}_{L^\infty} T \vLip \TV{\overline{\rho}} \right) \ell_n \rightarrow 0.
\end{align*}
We now estimate the error terms. Using \eqref{e:Ri-Rim_est}, \eqref{e:x_i_dot_est}, \eqref{e:bound_R_i} we get the estimate for $\mathfrak{E}^n_1$
\begin{align*}
\abs{\mathfrak{E}^n_1}
\leqslant{}&
\frac{\tau_n}{\ell_n} \vLip \left( \norm{\partial_t\varphi}_{L^\infty} + V\norm{\partial_x\varphi}_{L^\infty} \right)
\\
&\times \sum_{m=0}^{M_n-1} \sum_{i=0}^{N_n-1} \int_{t^{m}}^{t^{m+1}}
R_{i}(t) \, R_{i}^{m} \left[ \theta \abs{R_{i+1}^{m}-R_{i}^{m}}
+\left(1-\theta\right) \abs{R_{i+1}^{m+1}-R_{i}^{m+1}} \right] \left(x_{i+1}^{m}-x_{i}^{m}\right) \d{t}
\\
\leqslant{}&
\left( \vLip \left( \norm{\partial_t\varphi}_{L^\infty} + V \norm{\partial_x\varphi}_{L^\infty} \right) T \, R \, \TV{\overline{\rho}} \right) \tau_n \rightarrow 0,
\end{align*}
where we have used Proposition~\ref{prop:BV_estimate}. In order to control $\mathfrak{E}^n_2$ we use summation by parts in the index $i$ as follows:
\begin{align*}
\abs{\mathfrak{E}^n_2} &= \abs{\sum_{m=0}^{M_n-1} \sum_{i=0}^{N_n-1} \int_{t^{m}}^{t^{m+1}} R_{i}(t) \left(\int_{x_{i+1}(t)}^{x_{i+1}^{m}} \partial_t\varphi(x,t) \, \d{x} -\int_{x_{i}(t)}^{x_{i}^{m}} \partial_t\varphi(x,t) \, \d{x} \right) \d{t}}\\
&= \left|\sum_{m=0}^{M_n-1} \int_{t^{m}}^{t^{m+1}} \right.
\begin{aligned}[t]
\left[ \vphantom{\int_{x_{N_n}(t)}^{x_{N_n}^{m}}} \right.&
\sum_{i=1}^{N_n-1} \left(R_{i-1}(t)-R_{i}(t)\right) \int_{x_{i}(t)}^{x_{i}^{m}} \partial_t\varphi(x,t) \, \d{x}
\\&\left.\left.
+ R_{N_n-1}(t)\int_{x_{N_n}(t)}^{x_{N_n}^{m}} \partial_t\varphi(x,t) \, \d{x} -R_0(t)\int_{x_0(t)}^{x_0^{m}} \partial_t\varphi(x,t) \, \d{x} \right] \d{t} \right|
\end{aligned}\\
&\leqslant \left( T \, \TV{\overline{\rho}} \norm{ \partial_t\varphi}_{L^\infty} V + 2 \, T \, R \norm{\partial_t\varphi}_{L^\infty} V \right) \tau_n \rightarrow 0,
\end{align*}
with the last inequality following from \eqref{e:x_i_minus_x_m_est}, \eqref{e:bound_R_i} and from Proposition~\ref{prop:BV_estimate}. We use summation by parts in $i$ in the estimate for $\mathfrak{E}^n_3$ as well as follows:
\begin{align*}
\mathfrak{E}^n_3 ={}& \sum_{m=0}^{M_n-1} \sum_{i=0}^{N_n-1} \int_{t^{m}}^{t^{m+1}}R_{i}(t) \, \dot{x}_{i}(t)\left( \left( \varphi\!\left(x_{i+1}^{m},t\right) - \varphi\!\left(x_{i+1}(t),t\right) \right) - \left( \varphi\!\left(x_{i}^{m},t\right) - \varphi\!\left(x_{i}(t),t\right) \right) \right) \d{t}
\\={}&
\sum_{m=0}^{M_n-1} \int_{t^{m}}^{t^{m+1}}
\left[ \vphantom{\sum_{i=1}^{N_n-1}} \right.
\sum_{i=1}^{N_n-1} \left(R_{i-1}(t) \, \dot{x}_{i-1}(t)-R_{i}(t) \, \dot{x}_{i}(t)\right) \left(\varphi\!\left(x_{i}^{m},t\right) - \varphi\!\left(x_{i}(t),t\right)\right)
\\&
\left. \vphantom{\sum_{i=1}^{N_n-1}}
+R_{N_n-1}(t) \, \dot{x}_{N_n-1}(t)\left(\varphi(x_{N_n}^{m},t) - \varphi(x_{N_n}(t),t)\right)
- R_0(t) \, \dot{x}_0(t)\left(\varphi(x_0^{m},t) -\varphi(x_0(t),t)\right)\right] \d{t}
\\={}&
\sum_{m=0}^{M_n-1} \int_{t^{m}}^{t^{m+1}}
\left[ \vphantom{\sum_{i=1}^{N_n-1}} \right.
\sum_{i=1}^{N_n-1}
\left(R_{i-1}(t)-R_{i}(t)\right) \dot{x}_{i-1}(t) \left(\varphi\!\left(x_{i}^{m},t\right) - \varphi\!\left(x_{i}(t),t\right)\right)
\\&
+ \sum_{i=1}^{N_n-1}
R_{i}(t) \left( \dot{x}_{i-1}(t) - \dot{x}_{i}(t)\right)
\left(\varphi\!\left(x_{i}^{m},t\right) - \varphi\!\left(x_{i}(t),t\right)\right)
\\&\left. \vphantom{\sum_{i=1}^{N_n-1}}
+R_{N_n-1}(t) \, \dot{x}_{N_n-1}(t)\left(\varphi(x_{N_n}^{m},t) - \varphi(x_{N_n}(t),t)\right)
-R_0(t) \, \dot{x}_0(t) \left(\varphi(x_0^{m},t) -\varphi(x_0(t),t)\right) \right] \d{t}
\end{align*}
which leads, thanks to \eqref{e:x_i_minus_x_m_est} and to \eqref{e:x_i_dot_est}, to the estimate
\begin{align*}
&\abs{\mathfrak{E}^n_3} \leqslant
T \left(
\TV{\overline{\rho}} \norm{\partial_x\varphi}_{L^\infty} V^2
+ R \, 2 \, V \norm{\partial_x\varphi}_{L^\infty} V
+ 2 \, R \, V \norm{\partial_x\varphi}_{L^\infty} V
\right)\tau_n \rightarrow 0.
\end{align*}
As for the last term, observe first that by \eqref{e:x_dot} we have
\[\abs{ v\!\left(R_{i}^{m}\right)-\dot{x}_{i}(t) } =
\left(1-\theta\right) \abs{ v\!\left(R_{i}^{m+1}\right) - v\!\left(R_{i}^{m}\right) } \leqslant
\left(1-\theta\right) \vLip \abs{ R_{i}^{m+1} - R_{i}^{m} } .\]
Furthermore, we have
\[\abs{ R_{i}^{m} (\varphi\!\left(x_{i+1}^{m},t\right) - \varphi\!\left(x_{i}^{m},t\right)) }
\leqslant
R_{i}^{m} \norm{\partial_x\varphi}_{L^\infty} \left(x_{i+1}^{m}-x_{i}^{m}\right) =
\ell_{N_n } \norm{\partial_x\varphi}_{L^\infty} .\]
Moreover, by \eqref{e:Rdot1} we have
\begin{equation}
\label{e:R_i_time_est}
\abs{R_{i}^{m+1}-R_{i}^{m}} \leqslant \frac{\tau_n}{\ell_n} \vLip R_i^{m} \, R_i^{m+1} \left[ \theta \abs{R_{i+1}^{m}-R_{i}^{m}} +\left(1-\theta\right) \abs{R_{i+1}^{m+1}-R_{i}^{m+1}} \right] ,
\end{equation}
and by Propositions~\ref{prop:max_princ} and~\ref{prop:BV_estimate} it follows
\[\sum_{i=0}^{N_n-1} \abs{R_{i}^{m+1}-R_{i}^{m}} \leqslant \frac{\tau_n}{\ell_n} \vLip R^2 \, \TV{\overline{\rho}} .\]
The above estimates together with \eqref{e:R_i_time_est} and Proposition~\ref{prop:BV_estimate} lead to
\begin{align*}
\abs{\mathfrak{E}^n_4} &\leqslant \left(1-\theta\right) \vLip \ell_n \norm{\partial_x\varphi}_{L^\infty} \sum_{m=0}^{M_n-1} \sum_{i=0}^{N_n-1} \int_{t^{m}}^{t^{m+1}} \abs{R_{i}^{m+1}-R_{i}^{m}} \d{t}
\\
&= \left(1-\theta\right) \vLip \ell_n \norm{\partial_x\varphi}_{L^\infty} \, \tau_n \sum_{m=0}^{M_n-1} \sum_{i=0}^{N_n-1} \abs{R_{i}^{m+1}-R_{i}^{m}}\\
& 
\leqslant \left( \left(1-\theta\right) \vLip^2 \norm{\partial_x\varphi}_{L^\infty} T \, R^2 \, \TV{\overline{\rho}} \right) \tau_n
\rightarrow 0.
\end{align*}
This concludes the proof.
\end{proof}

\subsection{Convergence to the entropy solution}

In order to complete the proof of Theorem~\ref{thm:main}, we are only left with proving that if $M_n,N_n\nearrow +\infty$ subject to \eqref{e:CFL}, \eqref{e:CFL-extra} and $(\rho_n)_n$ converges strongly in $L^1(\R\times [0,T])$ and almost everywhere to $\rho$, then $\rho$ is the unique entropy solution to the Cauchy problem \eqref{e:CP} in the sense of Definition~\ref{d:entro}.

\begin{prop}
Let $(N_n)_n,(M_n)_n\subseteq\N$ be such that $M_n,N_n \nearrow +\infty$ and satisfying \eqref{e:CFL-extra} and the CFL-type condition \eqref{e:CFL}.
Assume that $(\rho_n)_n$ as in \eqref{e:rn} converges to some limit $\rho\in L^\infty_c(\R\times [0,T])$ in $L^1(\R\times[0,T])$ and that it satisfies \eqref{e:SleepToken}.
Then, $\rho$ is an entropy solution to the Cauchy problem \eqref{e:CP} on $\R\times [0,T]$.
\end{prop}

\begin{proof}
We first observe that \eqref{e:CFL-extra} implies that
\begin{equation*}
\lim_{n\to+\infty} \frac{\tau_n}{\ell_n}=0,
\end{equation*}
where $\tau_n = T/M_n$ and $\ell_n =L/N_n$.
Let $\varphi\in C_c^\infty(\R\times [0,T))$ with $\varphi\geqslant 0$, and let $k\in [0,R]$. Using an argument similar to the one at the beginning of the proof of Proposition~\ref{prop:weak}, it is sufficient to prove that $\lim_{n\to+\infty} \mathcal{E}^n \geqslant 0$, where
\begin{align*}
\mathcal{E}^n ={}& \int_0^T\int_\R \left[ \abs{\rho_n(x,t)-k} \partial_t\varphi(x,t)+\Phi_k\left(\rho_n(x,t)\right) \partial_x\varphi(x,t)\right] \d{x} \, \d{t}
+ \int_\R \abs{\rho_n(x,0)-k} \varphi(x,0) \, \d{x} ,
\end{align*}
with $k\in[0,R]$ and
\begin{align*}
&\Phi_k(\rho) = \sgn{\rho-k} \left( f(\rho)-f(k) \right),&
&f(\rho) = \rho \, v(\rho).
\end{align*}
We rewrite $\mathcal{E}^n$ by involving the piecewise reconstruction $R_i(t)$ given in \eqref{e:Ri} in order to obtain
\begin{align*}
\mathcal{E}^n ={}&
\sum_{m=0}^{M_n-1} \int_{t^{m}}^{t^{m+1}} \left[ \int_{-\infty}^{x_0^{m}} \left(k \, \partial_t\varphi(x,t) + f(k) \, \partial_x\varphi(x,t)\right) \d{x}
+ \int_{x_{N_n}^{m}}^{+\infty} \left(k \, \partial_t\varphi(x,t) + f(k) \, \partial_x\varphi(x,t)\right) \d{x} \right] \d{t}
\\&
+ \sum_{m=0}^{M_n-1} \sum_{i=0}^{N_n-1} \int_{t^{m}}^{t^{m+1}}
\left[
\abs{R_{i}^{m}-k} \int_{x_{i}^{m}}^{x_{i+1}^{m}} \partial_t\varphi(x,t) \, \d{x}
+\Phi_k\left(R_{i}^{m}\right) \left(\varphi\!\left(x_{i+1}^{m},t\right) - \varphi\!\left(x_{i}^{m},t\right)\right) \right] \d{t}
\\&
+ k \int_{-\infty}^{x^0_0} \varphi(x,0) \, \d{x}
+ k \int_{x_{N_n}^0}^{+\infty} \varphi(x,0) \, \d{x}
+ \sum_{i=0}^{N_n-1} \abs{R_{i}^0-k} \int_{x_{i}^0}^{x_{i+1}^0} \varphi(x,0) \, \d{x}
\\={}&
\mathfrak{M}^n_1 + \mathfrak{M}^n_2 + \mathfrak{E}^n_1 + \mathfrak{E}^n_2 + \mathfrak{E}^n_3 + \mathfrak{E}^n_4 + \mathfrak{E}^n_5,
\end{align*}
where
\begin{align*}
\mathfrak{M}^n_1 ={}&
\int_0^T \int_{-\infty}^{x_0(t)} \left(k \, \partial_t\varphi(x,t) + f(k) \partial_x\varphi(x,t)\right) \d{x} \, \d{t} + \int_0^T\int_{x_N(t)}^{+\infty} \left(k \, \partial_t\varphi(x,t) + f(k) \, \partial_x\varphi(x,t)\right) \d{x} \, \d{t}
\\&
+ k \int_{-\infty}^{x_0(0)} \varphi(x,0) \, \d{x}
+ k \int_{x_{N_n}(0)}^{+\infty} \varphi(x,0) \, \d{x} ,
\\
\mathfrak{M}^n_2 ={}&
\sum_{i=0}^{N_n-1} \int_{0}^{T} \left[
\abs{R_{i}(t)-k} \int_{x_{i}(t)}^{x_{i+1}(t)} \partial_t\varphi(x,t) \, \d{x}
+\Phi_k\left(R_{i}(t)\right) \left(\varphi\!\left(x_{i+1}(t),t\right) - \varphi\!\left(x_{i}(t),t\right)\right) \right] \d{t}
\\&
+\sum_{i=0}^{N_n-1} \abs{R_{i} (0)-k} \int_{x_{i} (0)}^{x_{i+1} (0)} \varphi(x,0) \, \d{x} ,
\end{align*}
and
\begin{align*}
\mathfrak{E}^n_1 ={}&
\sum_{m=0}^{M_n-1} \sum_{i=0}^{N_n-1} \int_{t^{m}}^{t^{m+1}} \left( \abs{R_{i}^{m}-k} - \abs{R_{i}(t)-k} \right) \left( \int_{x_{i}^{m}}^{x_{i+1}^{m}} \partial_t\varphi(x,t) \, \d{x} \right) \d{t} , \\
\mathfrak{E}^n_2 ={}&
\sum_{m=0}^{M_n-1} \sum_{i=0}^{N_n-1} \int_{t^{m}}^{t^{m+1}} \left( \Phi_k\left(R_{i}^{m}\right) - \Phi_k\left(R_{i}(t)\right) \right) \left(\varphi\!\left(x_{i+1}^{m},t\right) - \varphi\!\left(x_{i}^{m},t\right)\right) \d{t}, \\
\mathfrak{E}^n_3 ={}&
\sum_{m=0}^{M_n-1} \sum_{i=0}^{N_n-1} \int_{t^{m}}^{t^{m+1}} \abs{R_{i}(t)-k} \left(\int_{x_{i}^{m}}^{x_{i+1}^{m}} \partial_t\varphi(x,t) \, \d{x} - \int_{x_{i}(t)}^{x_{i+1}(t)} \partial_t\varphi(x,t) \, \d{x} \right) \d{t} , \\
\mathfrak{E}^n_4 ={}&
\sum_{m=0}^{M_n-1} \sum_{i=0}^{N_n-1} \int_{t^{m}}^{t^{m+1}} \Phi_k\left(R_{i}(t)\right) \left(\varphi\!\left(x_{i+1}^{m},t\right) - \varphi\!\left(x_{i}^{m},t\right) - \varphi\!\left(x_{i+1}(t),t\right)+\varphi\!\left(x_{i}(t),t\right)\right) \d{t} , \\
\mathfrak{E}^n_5 ={}&
\sum_{m=0}^{M_n-1} \int_{t^{m}}^{t^{m+1}} \left[ \int_{x_0(t)}^{x_0^{m}} \left(k \, \partial_t\varphi(x,t) + f(k) \, \partial_x\varphi(x,t)\right) \d{x}
+ \int_{x_N^{m}}^{x_N(t)} \left(k \, \partial_t\varphi(x,t) + f(k) \, \partial_x\varphi(x,t)\right) \d{x} \right] \d{t}.
\end{align*}
Let us first evaluate the first main term $\mathfrak{M}^n_1$ as follows. Integration by parts in time and the equality $\dot{x}_{N_n}(t) = v(0)$ imply
\begin{align*}
\mathfrak{M}^n_1 ={}& k \int_0^T\left[\frac{\d}{\d{t}} \left(\int_{-\infty}^{x_0(t)} \varphi(x,t) \, \d{x} \right) - \dot{x}_0(t) \, \varphi(x_0(t),t)+v(k) \, \varphi(x_0(t),t)\right] \d{t}
\\
&+ k \int_0^T\left[\frac{\d}{\d{t}} \left(\int_{x_{N_n}(t)}^{+\infty} \varphi(x,t) \, \d{x} \right) + \dot{x}_{N_n}(t) \, \varphi(x_{N_n}(t),t)-v(k) \, \varphi(x_{N_n}(t),t)\right] \d{t}
\\
&+k\int_{-\infty}^{x_0(0)} \varphi(x,0) \, \d{x} + k \int_{x_{N_n}(0)}^{+\infty} \varphi(x,0) \, \d{x}
\\
={}& k \int_0^T \left( \varphi(x_0(t),t)\left(v(k)-\dot{x}_0(t)\right) - \varphi(x_{N_n}(t),t)\left(v(k)-\dot{x}_{N_n}(t)\right) \right) \d{t}
\\
={}& \widetilde{\mathfrak{M}}^n_1 + \mathfrak{E}^n_6,
\end{align*}
with
\begin{align} \label{e:M1_est}
\widetilde{\mathfrak{M}}^n_1 &= k \int_0^T \left( \varphi(x_0(t),t) \left(v(k)-v\!\left(R_0(t)\right)\right) - \varphi(x_{N_n}(t),t) (v(k) - v(0)) \right) \d{t} ,
\\\nonumber
\mathfrak{E}^n_6 &= k \int_0^T \varphi(x_0(t),t) \left( v\!\left(R_0(t)\right) - \dot{x}_0(t) \right) \d{t} .
\end{align}
With a similar strategy, we write $\mathfrak{M}^n_2$ as follows:
\begin{align*}
\mathfrak{M}^n_2 ={}& \sum_{i=0}^{N_n-1} \int_0^T
\left[ \vphantom{\int_{x_{i}(t)}^{x_{i+1}(t)}} \right.
\begin{aligned}[t]
&\abs{R_{i}(t)-k} \left(\frac{\d}{\d{t}} \left(\int_{x_{i}(t)}^{x_{i+1}(t)} \varphi(x,t) \, \d{x} \right) - \varphi\!\left(x_{i+1}(t),t\right)\dot{x}_{i+1}(t)+\varphi\!\left(x_{i}(t),t\right)\dot{x}_{i}(t)\right)
\\&\left. \vphantom{\int_{x_{i}(t)}^{x_{i+1}(t)}}
+\Phi_k\left(R_{i}(t)\right) (\varphi\!\left(x_{i+1}(t),t\right) - \varphi\!\left(x_{i}(t),t\right))\right] \d{t}
\end{aligned}
\\
&+ \sum_{i=0}^{N_n-1} \abs{R_{i} (0)-k} \int_{x_{i} (0)}^{x_{i+1} (0)} \varphi(x,0) \, \d{x}
\\={}&
\sum_{i=0}^{N_n-1} \int_0^T \sgn{R_{i}(t)-k}
\left[ \vphantom{\int_{x_{i}(t)}^{x_{i+1}(t)}} \right.
\begin{aligned}[t]
&-\dot{R}_{i}(t)\int_{x_{i}(t)}^{x_{i+1}(t)} \varphi(x,t) \, \d{x}
\\
&+\left[R_{i}(t)(v\!\left(R_{i}(t)\right) - \dot{x}_{i+1}(t))-k(v(k) - \dot{x}_{i+1}(t))\right]\varphi\!\left(x_{i+1}(t),t\right)
\\
&\left. \vphantom{\int_{x_{i}(t)}^{x_{i+1}(t)}}
- \left[R_{i}(t)(v\!\left(R_{i}(t)\right)-\dot{x}_{i}(t)) - k(v(k) - \dot{x}_{i}(t))\right]\varphi\!\left(x_{i}(t),t\right) \right] \d{t}
\end{aligned}
\\={}&
\sum_{i=0}^{N_n-1} \int_0^T \sgn{R_{i}(t)-k}
\left[ \vphantom{\int_{x_{i}(t)}^{x_{i+1}(t)}} \right.
\begin{aligned}[t]
&\frac{R_{i}(t)^2}{\ell_n} (\dot{x}_{i+1}(t) - \dot{x}_{i}(t)) \int_{x_{i}(t)}^{x_{i+1}(t)} \varphi(x,t) \, \d{x}
\\&
+ \frac{R_{i}(t)^2}{\ell_n} (v\!\left(R_{i}(t)\right) - \dot{x}_{i+1}(t)) \int_{x_{i}(t)}^{x_{i+1}(t)} \varphi\!\left(x_{i+1}(t),t\right) \d{x}
\\&
- \frac{R_{i}(t)^2}{\ell_n} (v\!\left(R_{i}(t)\right)-\dot{x}_{i}(t)) \int_{x_{i}(t)}^{x_{i+1}(t)} \varphi\!\left(x_{i}(t),t\right) \d{x}
\\&\left. \vphantom{\int_{x_{i}(t)}^{x_{i+1}(t)}}
+ k \left( \dot{x}_{i+1}(t) - v(k) \right) \varphi\!\left(x_{i+1}(t),t\right)
+ k \left( v(k) - \dot{x}_{i}(t) \right) \varphi\!\left(x_{i}(t),t\right) \right] \d{t}
\end{aligned}
\\
={}& \widetilde{\mathfrak{M}}^n_2+ \mathfrak{E}^n_7 + \mathfrak{E}^n_8,
\end{align*}
with
\begin{align}
\label{e:M2_est}
\widetilde{\mathfrak{M}}^n_2 ={}&
k \sum_{i=0}^{N_n-1} \int_0^T \sgn{R_i (t)- k} \left[ (v (R_{i+1}(t)) - v(k)) \, \varphi\!\left(x_{i+1}(t),t\right) + \left(v(k)-v\!\left(R_i(t)\right)\right) \varphi\!\left(x_i(t),t\right) \right] \d{t},
\\\nonumber
\mathfrak{E}^n_7 ={}& \sum_{i=0}^{N_n-1} \int_0^T \sgn{R_{i}(t)-k} \frac{R_{i}(t)^2}{\ell_n} \left[ \vphantom{\int_{x_{i}(t)}^{x_{i+1}(t)}} \right.
(\dot{x}_{i+1}(t) - \dot{x}_{i}(t))\int_{x_{i}(t)}^{x_{i+1}(t)} \varphi(x,t) \, \d{x}
\\\nonumber
&\left. \vphantom{\int_{x_{i}(t)}^{x_{i+1}(t)}}
-(\dot{x}_{i+1}(t)-v\!\left(R_{i}(t)\right) \int_{x_{i}(t)}^{x_{i+1}(t)} \varphi\!\left(x_{i+1}(t),t\right) \d{x} +(\dot{x}_{i}(t)-v\!\left(R_{i}(t)\right)\int_{x_{i}(t)}^{x_{i+1}(t)} \varphi\!\left(x_{i}(t),t\right) \d{x}
\right] \d{t},
\\\nonumber
\mathfrak{E}^n_8 ={}& k \sum_{i=0}^{N_n-1} \int_0^T \sgn{R_i(t) - k} \left[ \left( \dot{x}_{i+1}(t) - v\!\left(R_{i+1}(t)\right) \right) \varphi( x_{i+1}(t),t) + \left(v (R_i (t)) - \dot{x}_i (t) \right) \varphi\!\left(x_i(t),t\right) \right] \d{t}.
\end{align}
We now estimate the terms $\mathfrak{E}^n_i$, $i=\disint{1}{8}$, and show that they are vanishing as $n\to+\infty$.
Concerning $\mathfrak{E}^n_1$, by using that $\abs{ \abs{a-c} - \abs{b-c}} \leqslant \abs{a-b}$ for all $a,b,c \in \R$, \eqref{e:Ri-Rim_est}, \eqref{e:bound_R_i} and Proposition~\ref{prop:BV_estimate} we get
\begin{align*}
\abs{\mathfrak{E}^n_1} & \leqslant \sum_{m=0}^{M_n-1} \sum_{i=0}^{N_n-1} \int_{t^{m}}^{t^{m+1}} \abs{R_i^{m} - R_i (t)} \left( \int_{x_i^{m}}^{x_{i+1}^{m}} \abs{\de_t\varphi(x,t)} \d{x} \right) \d{t}
\\&
\leqslant \sum_{m=0}^{M_n-1} \sum_{i=0}^{N_n-1} \int_{t^{m}}^{t^{m+1}}
\frac{\tau_n}{\ell_n} \vLip R \, R_{i}^{m} \left[ \theta \abs{R_{i+1}^{m}-R_{i}^{m}} +\left(1-\theta\right) \abs{R_{i+1}^{m+1}-R_{i}^{m+1}} \right]
\norm{\de_t \varphi}_{\infty} \left(x_{i+1}^{m} - x_i^{m}\right) \d{t}
\\&
\leqslant T \sum_{i=0}^{N_n-1} \tau_n \vLip R \left[ \theta \abs{R_{i+1}^{m}-R_{i}^{m}} +\left(1-\theta\right) \abs{R_{i+1}^{m+1}-R_{i}^{m+1}} \right] \norm{\de_t \varphi}_{\infty}
\\&
\leqslant \left( T \vLip R \, \TV{\overline{\rho}} \norm{\de_t \varphi}_{\infty} \right) \tau_n \to 0.
\end{align*}
We use that $[\Phi_k]_{\mathrm{Lip}} = \norm{f'}_{\infty}$, and by \eqref{e:Ri-Rim_est}, \eqref{e:bound_R_i} and Proposition~\ref{prop:BV_estimate} we estimate
\begin{align*}
\abs{\mathfrak{E}^n_2} & \leqslant \norm{f'}_{\infty} \sum_{m=0}^{M_n-1} \sum_{i=0}^{N_n-1} \int_{t^{m}}^{t^{m+1}} \abs{R_i^{m} - R_i(t)} \norm{\de_x \varphi}_{\infty} \left(x_{i+1}^{m} - x_i^{m}\right) \d{t} \\
& \leqslant \norm{f'}_{\infty} \sum_{m=0}^{M_n-1} \sum_{i=0}^{N_n-1} \int_{t^{m}}^{t^{m+1}} \tau_n \vLip R \left( \theta \abs{R_{i+1}^{m} - R_i^{m}} + \left(1-\theta\right) \abs{R_{i+1}^{m+1} - R_i^{m+1}} \right) \norm{\de_x \varphi}_{\infty} \, \d{t}
\\&
\leqslant \left( \norm{f'}_{\infty} T \vLip R \, \TV{\overline{\rho}} \norm{\de_x \varphi}_{\infty} \right) \tau_n \to 0.
\end{align*}
Concerning $\mathfrak{E}^n_3$, by \eqref{e:x_i_minus_x_m_est} we get
\begin{align*}
\abs{\mathfrak{E}^n_3} & = \abs{ \sum_{m=1}^{M_n-1} \sum_{i=0}^{N_n-1} \int_{t^{m}}^{t^{m+1}} \abs{ R_i (t)-k } \left( \int_{x_{i+1}(t)}^{x_{i+1}^{m}} \de_t \varphi\, \d{x} - \int_{x_i(t)}^{x_i^{m}} \de_t \varphi\, \d{x} \right) \d{t} }
\\
& \leqslant \sum_{m=0}^{M_n-1} \sum_{i=0}^{N_n-1} \int_{t^{m}}^{t^{m+1}} R \norm{\de_t \varphi}_{\infty} (\abs{x_{i+1}(t)-x_{i+1}^{m}} + \abs{x_i(t)-x_i^{m}}) \, \d{t} \\
& \leqslant M_n \, N_n \, \tau_n \, R \norm{\de_t \varphi}_{\infty} 2 \, \tau_n \, V
\leqslant \left( 2 \, T \, L \, R \norm{\de_t \varphi}_{\infty} V \right) \frac{\tau_n}{\ell_n} \to 0.
\end{align*}
As for the previous term, by \eqref{e:bound_R_i}, \eqref{e:x_i_minus_x_m_est} we have
\begin{align*}
\abs{\mathfrak{E}^n_4} & = \abs{ \sum_{m=0}^{M_n-1} \sum_{i=0}^{N_n-1} \int_{t^{m}}^{t^{m+1}} \Phi_k(R_i(t)) \left[ \varphi(x_{i+1}^{m}, t) - \varphi\!\left(x_{i+1}(t),t\right) - (\varphi\!\left(x_{i}^{m},t\right) - \varphi\!\left(x_i(t),t\right)) \right] \d{t} } \\
& \leqslant \sum_{m=0}^{M_n-1} \sum_{i=0}^{N_n-1} \int_{t^{m}}^{t^{m+1}} \norm{f'}_{\infty} \abs{R_i(t)-k} \norm{\de_x \varphi}_{\infty} \left( \abs{x_{i+1}^{m} - x_{i+1}(t)} + \abs{x_i^{m} - x_i(t)} \right) \d{t} \\
& \leqslant \sum_{m=0}^{M_n-1} \sum_{i=0}^{N_n-1} \tau_n \norm{f'}_{\infty} R \norm{\de_x \varphi}_{\infty} 2 \tau_n V
= \left( 2 \, T \, V \, L \norm{f'}_{\infty} R \norm{\de_x \varphi}_{\infty} \right) \frac{\tau_n}{\ell_n} \to 0.
\end{align*}
Since $k>0$ and $v$ is a decreasing function, by \eqref{e:x_i_minus_x_m_est} we deduce
\begin{align*}
\abs{\mathfrak{E}^n_5} & = k \sum_{m=0}^{M_n-1} \int_{t^{m}}^{t^{m+1}} \left[ \int_{x_0(t)}^{x_0^{m}} (\de_t \varphi + v(k) \de_x \varphi) \, \d{x} + \int_{x_N^{m}}^{x_N(t)} (\de_t \varphi + v(k) \de_x \varphi ) \, \d{x} \right] \d{t} \\
& \leqslant k \sum_{m=0}^{M_n-1} \int_{t^{m}}^{t^{m+1}} \left( \norm{\de_t \varphi}_{\infty} + V \norm{\de_x \varphi}_{\infty} \right) \left(\abs{x_0^{m} - x_0 (t)} + \abs{x_N^{m} - x_N (t)}\right) \d{t} \\
& \leqslant \left( 2 \, k \, T (\norm{\de_t \varphi}_{\infty} + V \norm{\de_x \varphi}_{\infty}) V\right) \tau_n \to 0.
\end{align*}
By using \eqref{e:x_dot}, \eqref{e:Ri-Rim_est}, \eqref{e:R_i_time_est}, \eqref{e:bound_R_i} and \eqref{e:dot_minus_dot}, we estimate
\begin{align*}
\abs{\mathfrak{E}^n_7}  ={}& \left| \sum_{m=0}^{M_n-1}\sum_{i=0}^{N_n-1} \int_{t^{m}}^{t^{m+1}} \sgn{R_i (t)-k} \frac{R_i (t)^2}{\ell_n} \left[ (\dot{x}_{i+1}(t) - \dot{x}_i(t) ) \int_{x_i(t)}^{x_{i+1}(t)} (\varphi(x,t) - \varphi\!\left(x_{i+1}(t),t\right)) \, \d{x} \right.\right.
\\& \left.\left.\qquad
+ \left(\dot{x}_i(t) - v\!\left(R_i(t)\right)\right) \int_{x_i(t)}^{x_{i+1}(t)} (\varphi(x_i (t),t) - \varphi\!\left(x_{i+1}(t),t\right)) \, \d{x} \right] \d{t} \right|
\\={}&
\left| \sum_{m=0}^{M_n-1}\sum_{i=0}^{N_n-1} \int_{t^{m}}^{t^{m+1}} \sgn{R_i (t)-k} \frac{R_i (t)^2}{\ell_n}
\left[ \vphantom{\int_{x_{i}(t)}^{x_{i+1}(t)}} \right.
(\dot{x}_{i+1}(t) - \dot{x}_i(t) ) \int_{x_i(t)}^{x_{i+1}(t)} (\varphi(x,t) - \varphi\!\left(x_{i+1}(t),t\right)) \, \d{x} \right.
\\& \left.\left.\ 
+ \left[ \left(1-\theta\right) \left(v\!\left(R_i^{m+1}\right)-v\!\left(R_i^{m}\right)\right) + v\!\left(R_i^{m}\right) - v\!\left(R_i(t)\right) \right] \int_{x_i(t)}^{x_{i+1}(t)} \left( \varphi(x_i (t),t) - \varphi\!\left(x_{i+1}(t),t\right) \right) \d{x} \right] \d{t} \right|
\\\leqslant{}&
\sum_{m=0}^{M_n-1} \sum_{i=0}^{N_n-1} \int_{t^{m}}^{t^{m+1}} \frac{R_i(t)^2}{\ell_n}
\left[ \vphantom{\int_{x_{i}(t)}^{x_{i+1}(t)}} \right.
\abs{\dot{x}_{i+1}(t) - \dot{x}_i(t)} \norm{\de_x \varphi}_{\infty} (x_{i+1}(t) - x_i(t))^2
\\
& \qquad
+ \left[ \left(1-\theta\right) \vLip \abs{R_i^{m+1} - R_i^{m}} + \vLip \abs{R_i^{m} - R_i(t)} \right] \norm{\de_x \varphi}_{\infty} (x_{i+1}(t) - x_i(t))^2 \left. \vphantom{\int_{x_{i}(t)}^{x_{i+1}(t)}} \right] \d{t}
\\\leqslant{}&
\sum_{m=0}^{M_n-1} \sum_{i=0}^{N_n-1} \int_{t^{m}}^{t^{m+1}} \ell_n \norm{\de_x \varphi}_{\infty}
\left[ \vphantom{\int_{x_{i}(t)}^{x_{i+1}(t)}} \right.
\abs{\dot{x}_{i+1}(t) - \dot{x}_i(t)} \\
& \qquad + \vLip
\left[ \vphantom{\frac{\tau_n}{\ell_n}} \right.
\begin{aligned}[t]
&\left(1-\theta\right) \frac{\tau_n}{\ell_n} \vLip R_i^{m} \, R_i^{m+1} \left[ \theta \abs{R_{i+1}^{m} - R_i^{m}} + \left(1-\theta\right) \abs{R_{i+1}^{m+1} - R_i^{m+1}} \right]
\\&
\left.
+ \frac{\tau_n}{\ell_n} \vLip R_i(t) R_i^{m} \left[ \theta \abs{R_{i+1}^{m} - R_i^{m}} + \left(1-\theta\right) \abs{R_{i+1}^{m+1} - R_i^{m+1}} \right]
\right]
\left. \vphantom{\int_{x_{i}(t)}^{x_{i+1}(t)}} \right] \d{t}
\end{aligned}
\\\leqslant{}&
\sum_{m=0}^{M_n-1} \sum_{i=0}^{N_n-1} \int_{t^{m}}^{t^{m+1}} \ell_n \norm{\de_x \varphi}_{\infty}
\left[ \vphantom{\int_{x_{i}(t)}^{x_{i+1}(t)}} \right.
\abs{\dot{x}_{i+1}(t) - \dot{x}_i(t)} \\
& \qquad + \vLip \left[
(2-\theta) \frac{\tau_n}{\ell_n} \vLip R^2 \left[ \theta \abs{R_{i+1}^{m} - R_i^{m}} + \left(1-\theta\right) \abs{R_{i+1}^{m+1} - R_i^{m+1}} \right]
\right]
\left. \vphantom{\int_{x_{i}(t)}^{x_{i+1}(t)}} \right] \d{t}
\\\leqslant{}&
\left( \norm{\de_x \varphi}_{\infty} T \vLip \TV{\overline{\rho}} \right) \ell_n
+ \left( \norm{\de_x \varphi}_{\infty} T (2-\theta) \vLip^2 R^2 \, \TV{\overline{\rho}} \right) \tau_n \to 0.
\end{align*}
Due to \eqref{e:x_dot}, \eqref{e:Ri-Rim_est}, \eqref{e:R_i_time_est}, \eqref{e:bound_R_i} and Proposition~\ref{prop:max_princ} we obtain
\begin{align*}
\abs{\mathfrak{E}^n_6}
& = k \abs{ \sum_{m=0}^{M_n-1} \int_{t^{m}}^{t^{m+1}} \varphi(x_0(t),t) \left[ \theta \left( v\!\left(R_0^{m+1}\right) - v\!\left(R_0^{m}\right) \right) + v\!\left(R_0(t)\right) - v\!\left(R_0^{m+1}\right) \right] \d{t} } \\
& \leqslant k \sum_{m=0}^{M_n-1} \tau_n \norm{\varphi}_{\infty} \vLip \left[ \theta \abs{R_0^{m+1} - R_0^{m}} + \abs{R_0(t) - R_0^{m}} + \abs{R_0^{m} - R_0^{m+1}} \right] \\
& \leqslant \left( k \, T \norm{\varphi}_{\infty} \vLip (\theta+2) \vLip R^3 \right) \frac{\tau_n}{\ell_n} \to 0.
\end{align*}
For the last error term, by \eqref{e:x_dot}, \eqref{e:Ri-Rim_est}, \eqref{e:R_i_time_est} and Proposition~\ref{prop:BV_estimate} we get
\begin{align*}
\abs{\mathfrak{E}^n_8}  \leqslant{}& k \sum_{i=0}^{N_n-1} \int_{0}^{T} \abs{
\left(\dot{x}_{i+1} (t) - v\!\left(R_{i+1}(t)\right)\right) \varphi\!\left(x_{i+1}(t),t\right)
- \left(\dot{x}_i(t) - v\!\left(R_i(t)\right)\right) \varphi\!\left(x_i(t),t\right) } \d{t}
\\={}&
k \sum_{m=0}^{M_n-1} \sum_{i=0}^{N_n-1} \int_{t^{m}}^{t^{m+1}}
\left| \vphantom{\int} \right.
\left( \left(1-\theta\right) \left( v\!\left(R_{i+1}^{m+1}\right) - v\!\left(R_{i+1}^{m}\right) \right) + \left( v\!\left(R_{i+1}^{m}\right) - v\!\left(R_{i+1}(t)\right) \right) \right) \varphi\!\left(x_{i+1}(t),t\right) \\
& \ - \left( \left(1-\theta\right) \left( v\!\left(R_i^{m+1}\right) - v\!\left(R_i^{m}\right) \right) + \left(v\!\left(R_i^{m}\right) - v\!\left(R_i(t)\right) \right) \right) \varphi\!\left(x_i(t),t\right) \left. \vphantom{\int} \right| \d{t}\\
\leqslant{}& k \sum_{m=0}^{M_n-1} \sum_{i=0}^{N_n-1} \int_{t^{m}}^{t^{m+1}} \norm{\varphi}_{\infty} \vLip
\left[ \vphantom{\int} \right.
\left(1-\theta\right) \abs{R_{i+1}^{m+1} - R_{i+1}^{m}} + \abs{R_{i+1}^{m}-R_{i+1}(t)}
\\
& \ +\left(1-\theta\right) \abs{R_i^{m+1} - R_i^{m}} + \abs{R_i^{m} - R_i(t)} 
\left. \vphantom{\int} \right] \d{t}
\\
\leqslant{}& k \sum_{m=0}^{M_n-1} \int_{t^{m}}^{t^{m+1}} \norm{\varphi}_{\infty} \vLip^2 \, \frac{\tau_n}{\ell_n} \, R^2 \, \TV{\overline{\rho}} 2(2-\theta) \, \d{t}\\
={}& \left( k \, T \norm{\varphi}_{\infty} \vLip^2 \, R^2 \, \TV{\overline{\rho}} 2 \, (2-\theta) \right) \frac{\tau_n}{\ell_n} \to 0.
\end{align*}
Then, the sum of $\widetilde{\mathfrak{M}}^n_1$ and $\widetilde{\mathfrak{M}}^n_2$ given in \eqref{e:M1_est} and \eqref{e:M2_est} reduces to
\begin{align*}
\widetilde{\mathfrak{M}}^n_1 + \widetilde{\mathfrak{M}}^n_2
={}&
k
\left[ \vphantom{\int_0^T} \right.
\int_0^T \left[\varphi(x_0(t),t) \left(v(k) - v\!\left(R_0(t)\right)\right) - \varphi(x_N(t),t) \left(v(k) - v(0)\right) \right] \d{t}
\\&
+ \sum_{i=0}^{N_n-1} \int_0^T \sgn{R_i (t) -k} \left[\left(v\!\left(R_{i+1}(t)\right) - v(k)\right) \varphi\!\left(x_{i+1}(t),t\right) + \left(v(k)-v\!\left(R_i(t)\right)\right) \varphi\!\left(x_i(t),t\right)\right] \d{t}
\left. \vphantom{\int_0^T} \right]
\\={}&
k \left[ \vphantom{\int_0^T} \right.
\int_0^T \varphi(x_0(t),t) \left(v(k)-v\!\left(R_0(t)\right)\right) \left(1+ \sgn{R_0(t)-k} \right) \d{t}
\\&
+ \int_0^T \varphi(x_N(t),t) \left(v(0)-v(k)\right) \left(1+ \sgn{R_{N-1}(t)-k}\right) \d{t}
\\&
+ \sum_{i=1}^{N_n-1} \int_0^T \varphi\!\left(x_i(t),t\right) \left(v\!\left(R_i(t)\right)-v(k)\right) \left(\sgn{R_{i-1}(t) -k} - \sgn{R_i(t)-k} \right) \d{t}
\left. \vphantom{\int_0^T} \right]
\geqslant 0,
\end{align*}
that completes the proof.
\end{proof}

Finally, we deduce the convergence of the entire sequence of approximate discrete densities $(\rho_n)_n$ by exploiting the uniqueness of the entropy solution and applying the Urysohn subsequence principle, see for instance \cite[(7.6)]{Coclite-book}.

\section*{Acknowledgments}
MDF and SF are partially supported by the Italian “National Centre for HPC, Big Data and Quantum Computing” - Spoke 5 “Environment and Natural Disasters” and by the Ministry of University and Research (MIUR) of Italy under the grant PRIN 2020- Project N. 20204NT8W4, Nonlinear Evolutions PDEs, fluid
dynamics and transport equations: theoretical foundations and applications.
MDF, SF and VI are partially supported by the InterMaths Network, \url{www.intermaths.eu}. MDF, SF and VI are also partially supported by the INdAM-GNAMPA project 2025 code CUP E5324001950001 ``Teoria e applicazioni dei modelli evolutivi:
trasporto ottimo, metodi variazionali e
approssimazioni particellari deterministiche''. SF and VI are partially  supported by the INdAM-GNAMPA project 2026 code CUP E53C25002010001 ``Modelli di reazione-diffusione-trasporto: dall'analisi alle applicazioni''.
MDR acknowledge financial support from {\em (a)}: the PRIN 2022 project \emph{Modeling, Control and Games through Partial Differential Equations} (D53D23005620006), funded by the European Union - Next Generation EU; {\em (b)}: the  INdAM - GNAMPA Research Project \emph{Modelli di traffico, di biologia e di dinamica dei gas basati su sistemi di equazioni iperboliche}, code CUP E5324001950001.

\appendix

\section{An interpolation inequality}\label{sec:appendix}

\begin{lem} \label{lem:appendix}
Fix $\mu,\eta\in \mathcal{P}_c(\R)\cap BV(\R)$. Then, there holds
\begin{equation*}
\norm{\mu-\eta}_{L^1} \leqslant 2\left(\TV{\mu}+\TV{\eta}\right)^{1/2}d_1(\mu,\eta)^{1/2},
\end{equation*}
where $d_1(\mu,\eta)$ is the $1$-Wasserstein distance between $\mu$ and $\eta$.
\end{lem}

\begin{proof}
Let $\xi_\varepsilon$ be a family of standard Friedrichs mollifiers, and set
\[\mu_\varepsilon = \mu\ast \xi_\varepsilon,\qquad \eta_\varepsilon = \eta \ast \xi_\varepsilon.\]
Clearly, $\mu_\varepsilon$ and $\eta_\varepsilon$ belong to the Sobolev space $W^{1,1} (\R)$. Moreover, \cite[Theorem 2.2]{ambrosio_fusco_pallara} implies for all $\varepsilon>0$
\begin{align} \label{e:AFP}
&\int_{\R} \abs{\mu_\varepsilon'(x)} \d{x} \leqslant \TV{\mu},&
&\int_{\R} \abs{\eta_\varepsilon'(x)} \d{x} \leqslant \TV{\eta}.
\end{align}
Now, we recall the interpolation inequality
\begin{equation} \label{e:interpolation}
\norm{u'}_{L^1} \leqslant 2\norm{u''}_{L^1}^{1/2} \norm{u}_{L^1}^{1/2},
\end{equation}
which holds for a compactly supported $u\in W^{2,1} (\R)$, see \cite[Theorem 6.4.1, Lemma 6.4.3]{hormander_hyperbolic_book}. Setting
\begin{align*}
&F_\varepsilon(x) = \int_{-\infty}^x \mu_\varepsilon(y) \, \d{y},&
&G_\varepsilon(x) = \int_{-\infty}^x \eta_\varepsilon(y) \, \d{y},
\end{align*}
the assumption of $\mu$ and $\eta$ being compactly supported implies that the difference $F_\varepsilon-G_\varepsilon$ belongs to $L^1(\R)$. Moreover, there holds
\[\norm{F_\varepsilon-G_\varepsilon}_{L^1} = d_1(\mu_\varepsilon,\eta_\varepsilon).\]
Choosing $u = F_\varepsilon-G_\varepsilon$ in \eqref{e:interpolation}, we get
\[\norm{\mu_\varepsilon-\eta_\varepsilon}_{L^1} \leqslant 2\norm{\mu_\varepsilon'-\eta'_\varepsilon}_{L^1}^{1/2}d_1(\mu_\varepsilon,\eta_\varepsilon)^{1/2} \leqslant 2\left(\norm{\mu_\varepsilon'}_{L^1}+\norm{\eta'_\varepsilon}_{L^1} \right)^{1/2}d_1(\mu_\varepsilon,\eta_\varepsilon)^{1/2} ,\]
and then \eqref{e:AFP} implies
\begin{align*}
& \norm{\mu_\varepsilon-\eta_\varepsilon}_{L^1} \leqslant 2(\TV{\mu}+\TV{\eta})^{1/2} d_1(\mu_\varepsilon,\eta_\varepsilon)^{1/2} .
\end{align*}
Since $\mu_\varepsilon\rightarrow \mu$ and $\eta_\varepsilon\rightarrow \eta$ locally strongly in $L^1(\R)$ and almost everywhere up to a subsequence, Fatou's lemma implies
\[\norm{\mu-\eta}_{L^1} \leqslant \liminf_{\varepsilon\searrow 0} \norm{\mu_\varepsilon-\eta_\varepsilon}_{L^1} .\]
We then observe that
\begin{multline*}
F_\varepsilon(x)-G_\varepsilon(x)
= \int_{-\infty}^x \left(\mu_\varepsilon(y) - \eta_\varepsilon(y)\right) \d{y}
= \int_{-\infty}^x \int_{-\infty}^{+\infty} \xi_\varepsilon(z)\left(\mu(y-z) - \eta(y-z)\right) \d{z} \, \d{y}\\
= \int_{-\infty}^{+\infty} \xi_\varepsilon(z) \int_{-\infty}^x\left(\mu(y-z) - \eta(y-z)\right) \d{y} \, \d{z} = \int_{-\infty}^{+\infty} \xi_\varepsilon(z) \left( F(x-z)-G(x-z)\right) \d{z}
= \xi_\varepsilon\ast (F-G),
\end{multline*}
where
\begin{align*}
&F(x) = \int_{-\infty}^x \mu(y) \, \d{y},&
&G(x) = \int_{-\infty}^x \eta(y) \, \d{y}.
\end{align*}
Hence, \cite[Theorem 2.2]{ambrosio_fusco_pallara} implies once again
\[d_1(\mu_\varepsilon,\eta_\varepsilon) = \norm{F_\varepsilon-G_\varepsilon}_{L^1} \leqslant \norm{F-G}_{L^1} = d_1(\mu,\eta),\]
which concludes the proof.
\end{proof}



\begin{thebibliography}{10}

\bibitem{ambrosio_fusco_pallara}
L.~Ambrosio, N.~Fusco, and D.~Pallara.
\newblock {\em Functions of bounded variation and free discontinuity problems}.
\newblock Oxford Mathematical Monographs. The Clarendon Press, Oxford University Press, New York, 2000.

\bibitem{AGS}
L.~Ambrosio, N.~Gigli, and G.~Savar\'e.
\newblock {\em Gradient flows in metric spaces and in the space of probability measures}.
\newblock Lectures in Mathematics ETH Z\"urich. Birkh\"auser Verlag, Basel, second edition, 2008.

\bibitem{AndreianovRosiniStivaletta}
B.~Andreianov, M.D. Rosini, and G.~Stivaletta.
\newblock On existence, stability and many-particle approximation of solutions of 1{D} {H}ughes' model with linear costs.
\newblock {\em J. Differential Equations}, 369:253--298, 2023.

\bibitem{bressan1992wft}
A.~Bressan.
\newblock Global solutions of systems of conservation laws by wave-front tracking.
\newblock {\em Journal of Mathematical Analysis and Applications}, 170(2):414--432, 1992.

\bibitem{bressan2000hyperbolic}
A.~Bressan.
\newblock {\em Hyperbolic Systems of Conservation Laws}.
\newblock Oxford University Press, 2000.

\bibitem{Coclite-book}
G.M. Coclite.
\newblock {\em Scalar conservation laws}.
\newblock SpringerBriefs in Mathematics. Springer, Singapore, [2024] \copyright 2024.

\bibitem{colombo_rossi}
R.M. Colombo and E.~Rossi.
\newblock On the micro-macro limit in traffic flow.
\newblock {\em Rend. Semin. Mat. Univ. Padova}, 131:217--235, 2014.

\bibitem{CrankNicolson1947}
J.~Crank and P.~Nicolson.
\newblock A practical method for numerical evaluation of solutions of partial differential equations of the heat-conduction type.
\newblock {\em Mathematical Proceedings of the Cambridge Philosophical Society}, 43(1):50--67, 1947.

\bibitem{dafermos1972polygonal}
C.~M. Dafermos.
\newblock Polygonal approximations of solutions of the initial value problem for a conservation law.
\newblock {\em Journal of Mathematical Analysis and Applications}, 38:33--41, 1972.

\bibitem{dafermos2016hyperbolic}
C.~M. Dafermos.
\newblock {\em Hyperbolic Conservation Laws in Continuum Physics}.
\newblock Springer, 4th edition, 2016.

\bibitem{DiFrancescoFagioliRadici}
M.~Di~Francesco, S.~Fagioli, and E.~Radici.
\newblock Deterministic particle approximation for nonlocal transport equations with nonlinear mobility.
\newblock {\em J. Differential Equations}, 266(5):2830--2868, 2019.

\bibitem{DiFrancescoFagioliRosiniRusso1}
M.~Di~Francesco, S.~Fagioli, M.D. Rosini, and G.~Russo.
\newblock Follow-the-leader approximations of macroscopic models for vehicular and pedestrian flows.
\newblock In {\em Active particles. {V}ol. 1. {A}dvances in theory, models, and applications}, Model. Simul. Sci. Eng. Technol., pages 333--378. Birkh\"auser/Springer, Cham, 2017.

\bibitem{DFIOSC2025}
M.~Di~Francesco, V.~Iorio, and M.~Schmidtchen.
\newblock The approximation of the quadratic porous medium equation via nonlocal interacting particles subject to repulsive {M}orse potential.
\newblock {\em SIAM Journal on Mathematical Analysis}, 57(5):4631--4679, 2025.

\bibitem{DiFrancescoRosini_ARMA}
M.~Di~Francesco and M.D. Rosini.
\newblock Rigorous derivation of nonlinear scalar conservation laws from follow-the-leader type models via many particle limit.
\newblock {\em Archive for Rational Mechanics and Analysis}, 217(3):831--871, 2015.

\bibitem{DiFrancescoStivaletta1}
M.~di~Francesco and G.~Stivaletta.
\newblock Convergence of the follow-the-leader scheme for scalar conservation laws with space dependent flux.
\newblock {\em Discrete Contin. Dyn. Syst.}, 40(1):233--266, 2020.

\bibitem{diperna1976global}
R.~J. DiPerna.
\newblock Global existence of solutions to nonlinear hyperbolic systems of conservation laws.
\newblock {\em Journal of Differential Equations}, 20(1):187--212, 1976.

\bibitem{FagioliFavre}
S.~Fagioli and G.~Favre.
\newblock Opinion formation on evolving network: the dpa method applied to a nonlocal cross-diffusion pde-ode system.
\newblock {\em European Journal of Applied Mathematics}, 35(6):748–775, 2024.

\bibitem{FagioliRadici}
S.~Fagioli and E.~Radici.
\newblock Solutions to aggregation–diffusion equations with nonlinear mobility constructed via a deterministic particle approximation.
\newblock {\em Mathematical Models and Methods in Applied Sciences}, 28(09):1801--1829, 2018.

\bibitem{FaTse}
S.~Fagioli and O.~Tse.
\newblock On gradient flow and entropy solutions for nonlocal transport equations with nonlinear mobility.
\newblock {\em Nonlinear Analysis}, 221:112904, 2022.

\bibitem{ferrari}
P.~A. Ferrari.
\newblock Shock fluctuations in asymmetric simple exclusion.
\newblock {\em Probab. Theory Related Fields}, 91(1):81--101, 1992.

\bibitem{ferrari_TASEP}
P.~L. Ferrari and P.~Nejjar.
\newblock Shock fluctuations in flat {TASEP} under critical scaling.
\newblock {\em J. Stat. Phys.}, 160(4):985--1004, 2015.

\bibitem{GaravelloPiccoli-book}
M.~Garavello and B.~Piccoli.
\newblock {\em Traffic flow on networks}, volume~1 of {\em AIMS Series on Applied Mathematics}.
\newblock American Institute of Mathematical Sciences (AIMS), Springfield, MO, 2006.
\newblock Conservation laws models.

\bibitem{gray2001granular}
J.~M. N.~T. Gray.
\newblock Granular flow in partially filled slowly rotating drums.
\newblock {\em Journal of Fluid Mechanics}, 441:1--29, 2001.

\bibitem{Haberman-book}
R.~Haberman.
\newblock {\em Mathematical models}, volume~21 of {\em Classics in Applied Mathematics}.
\newblock Society for Industrial and Applied Mathematics (SIAM), Philadelphia, PA, 1998.
\newblock Mechanical vibrations, population dynamics, and traffic flow, An introduction to applied mathematics, Reprint of the 1977 original.

\bibitem{HairerWannerII}
E.~Hairer, S.P. N{\o}rsett, and G.~Wanner.
\newblock {\em Solving Ordinary Differential Equations II: Stiff and Differential-Algebraic Problems}, volume~14 of {\em Springer Series in Computational Mathematics}.
\newblock Springer, Berlin, Heidelberg, 1996.

\bibitem{helbing2001traffic}
D.~Helbing.
\newblock Traffic and related self-driven many-particle systems.
\newblock {\em Reviews of Modern Physics}, 73:1067--1141, 2001.

\bibitem{holden2015front}
H.~Holden and N.~H. Risebro.
\newblock {\em Front Tracking for Hyperbolic Conservation Laws}, volume 152.
\newblock Springer, 2015.

\bibitem{holden2018continuum}
H.~Holden and N.~H. Risebro.
\newblock The continuum limit of follow-the-leader models --- a short proof.
\newblock {\em Discrete and Continuous Dynamical Systems}, 38(2):715--722, 2018.

\bibitem{holden2018ftl_lwr}
H.~Holden and N.~H. Risebro.
\newblock Follow-the-leader models can be viewed as a numerical approximation to the {L}ighthill--{W}hitham--{R}ichards model for traffic flow.
\newblock {\em Networks and Heterogeneous Media}, 13(3):409--421, 2018.

\bibitem{holden2024continuum_nonlocal_ftl}
H.~Holden and N.~H. Risebro.
\newblock The continuum limit of non-local follow-the-leader models.
\newblock {\em ESAIM: Mathematical Modelling and Numerical Analysis}, 58(4):1523--1539, 2024.

\bibitem{hormander_hyperbolic_book}
L.~H\"ormander.
\newblock {\em Lectures on nonlinear hyperbolic differential equations}, volume~26 of {\em Math\'ematiques \& Applications (Berlin) [Mathematics \& Applications]}.
\newblock Springer-Verlag, Berlin, 1997.

\bibitem{HundsdorferVerwer2003}
W.~Hundsdorfer and J.G. Verwer.
\newblock {\em Numerical Solution of Time-Dependent Advection-Diffusion-Reaction Equations}, volume~33 of {\em Springer Series in Computational Mathematics}.
\newblock Springer, Berlin, Heidelberg, 2003.

\bibitem{landim}
C.~Kipnis and C.~Landim.
\newblock {\em Scaling limits of interacting particle systems}, volume 320 of {\em Grundlehren der Mathematischen Wissenschaften [Fundamental Principles of Mathematical Sciences]}.
\newblock Springer-Verlag, Berlin, 1999.

\bibitem{kruvzkov}
S.~N. Kru\v{z}kov.
\newblock First order quasilinear equations in several independent variables.
\newblock {\em Mathematics of the USSR-Sbornik}, 10:\ 217--243, 1970.

\bibitem{kuznetsov1976accuracy}
N.~N. Kuznetsov.
\newblock Accuracy of some approximate methods for computing the weak solutions of a first-order quasilinear equation.
\newblock {\em USSR Computational Mathematics and Mathematical Physics}, 16(6):105--119, 1976.

\bibitem{kynch1952sedimentation}
G.~J. Kynch.
\newblock A theory of sedimentation.
\newblock {\em Transactions of the Faraday Society}, 48:166--176, 1952.

\bibitem{lighthill1955kinematic}
M.~J. Lighthill and G.~B. Whitham.
\newblock On kinematic waves {II.} {A} theory of traffic flow on long crowded roads.
\newblock {\em Proceedings of the Royal Society of London. Series A}, 229(1178):317--345, 1955.

\bibitem{lions_perthame_tadmor}
P.-L. Lions, B.~Perthame, and E.~Tadmor.
\newblock A kinetic formulation of multidimensional scalar conservation laws and related equations.
\newblock {\em J. American Math. Society}, 7:169–191, 1994.

\bibitem{QuarteroniValli2008}
A.~Quarteroni and A.~Valli.
\newblock {\em Numerical Approximation of Partial Differential Equations}, volume~23 of {\em Springer Series in Computational Mathematics}.
\newblock Springer, Berlin, Heidelberg, 2008.

\bibitem{richards1956shock}
P.~I. Richards.
\newblock Shock waves on the highway.
\newblock {\em Operations Research}, 4(1):42--51, 1956.

\bibitem{Rosini-book}
M.D. Rosini.
\newblock {\em Macroscopic models for vehicular flows and crowd dynamics: theory and applications}.
\newblock Understanding Complex Systems. Springer, Heidelberg, 2013.

\bibitem{savage1989motion}
S.~B. Savage and K.~Hutter.
\newblock The motion of a finite mass of granular material down a rough incline.
\newblock {\em Journal of Fluid Mechanics}, 199:177--215, 1989.

\bibitem{villani_optimal}
C.~Villani.
\newblock {\em Optimal transport. Old and new}, volume 338 of {\em Grundlehren der Mathematischen Wissenschaften}.
\newblock Springer-Verlag, Berlin, 2009.

\bibitem{whitham1974waves}
G.~B. Whitham.
\newblock {\em Linear and Nonlinear Waves}.
\newblock Wiley, New York, 1974.

\end{thebibliography}
\end{document}